\documentclass[a4paper, 12pt]{amsart}
\usepackage{latexsym, amssymb, amscd, amsfonts, amsthm, amsmath, graphicx}

\newcommand{\coeff}{\operatorname{coeff}}
\newcommand{\MT}{\operatorname{MT}}
\newcommand{\ET}{\operatorname{ET}}

\newcommand{\bC}{\mathbb{C}}
\newcommand{\bR}{\mathbb{R}}
\newcommand{\bN}{\mathbb{N}}

\newcommand{\sgn}{\mathrm{sgn}}
\newcommand{\Res}{\mathrm{Res}}

\newcommand{\E}{\mathbf{E}}
\newcommand{\Var}{\mathbf{Var}}

\newcommand{\vlambda}{{\mathbf{\lambda}}}
\newcommand{\vrho}{{\mathbf{\rho}}}

\newcommand{\vx}{{\mathbf{x}}}

\newcommand{\vz}{{\mathbf{z}}}
\newcommand{\ve}{{\mathbf{e}}}
\newcommand{\vk}{{\mathbf{k}}}

\newcommand{\cC}{{\mathcal{C}}}
\newcommand{\cB}{{\mathcal{B}}}

\newtheorem{theorem}{Theorem}

\newtheorem{lemma}[theorem]{Lemma}
\newtheorem{proposition}[theorem]{Proposition}

\newtheorem{conjecture}[theorem]{Conjecture}
\theoremstyle{remark}
\newtheorem*{rem}{Remark}
\newtheorem{definition}{Definition}

\title{The random $k$ cycle walk on the symmetric group}
\author{Bob Hough}
\address{Department of Mathematics, Stanford University, 450 Serra Mall,
Stanford, CA, 94305, USA}
\curraddr{Mathematical Institute, University of Oxford, 
Radcliffe Observatory Quarter, Woodstock Road, Oxford, OX2 6GG, UK.}
\email{hough@maths.ox.ac.uk}
\thanks{The author is grateful for financial support from a Ric Weiland
Graduate Research Fellowship. The final preparation of this manuscript was completed with support from
ERC Research Grant 279438, Approximate Algebraic Structure and Applications.}

\subjclass[2010]{Primary 60J10, 60B15, 20C30}
\keywords{Random walk on a group, character bounds, symmetric group, cut-off
phenomenon}

\begin{document}

\begin{abstract}
We study the random walk on the symmetric group $S_n$ generated by 
the conjugacy class of cycles of length $k$.  We show that the convergence
to uniform measure of this walk has a cut-off in total variation distance after
$\frac{n}{k}\log n$ steps, uniformly in $k = o(n)$ as $n \to \infty$.  The
analysis follows from a new asymptotic estimation of the characters of the
symmetric group evaluated at cycles. 
\end{abstract}

\maketitle

\section{Introduction}
A well-known conjecture in the theory of random walk on a group asserts that
for the
random walk on the symmetric group generated by all permutations of a common
cycle structure, the mixing time of the walk depends only on the number of fixed
points.  Given measure $\mu$ on $S_n$ and integer $t \geq 1$, denote $\mu^{*t}$ its $t$-fold convolution, which is the law of the random walk starting from the identity in $S_n$ after $t$ steps from $\mu$.  We measure convergence to uniformity in the total variation metric, which for probabilities $\mu$ and $\nu$ on $S_n$ is given by
\[
 \|\mu - \nu\|_{T.V.} = \max_{A \subset S_n} |\mu(A) - \nu(A)|.
\]
 A precise statement of the conjecture is as follows.

\begin{conjecture}\label{main_conjecture}
 For each $n > 1$ let $C_n$ be a conjugacy class of the symmetric group $S_n$
having $n-k_n< n$ fixed points, and let $\mu_{C_n}$ denote the uniform probability measure on $C_n$.  Let $t_2, t_3, t_4, ...$ be
a sequence of positive integer steps and for each $n \geq 2$ let $U_n$ be
uniform measure on the coset of
the alternating group $A_n$ supporting the measure $\mu_{C_n}^{*t_n}$.  For any
$\varepsilon > 0$, if eventually  $t_n
\geq (1 +\varepsilon) \frac{n}{k_n} \log n$, then
\[
\lim_{n \to \infty} \|\mu_{C_n}^{*t_n} - U_n\|_{T.V.} = 0.
\]
If eventually $t_n \leq (1-\varepsilon) \frac{n}{k_n}
\log n$, then
\[
 \lim_{n \to \infty} \|\mu_{C_n}^{*t_n} - U_n\|_{T.V.} = 1.
\]
\end{conjecture}

\noindent This
conjecture aims to generalize the mixing time analysis of Diaconis and
Shahshahani \cite{diaconis_shahshahani} for the random transposition walk.  The
formal conjecture seems to have first
appeared in \cite{roichman_1}.

When $k$ grows with $n$ like a constant times $n$ the
conjecture is known to be false  because the
walk mixes too rapidly.  This is the work of a number of authors, but first
\cite{lulov_pak} and for later results see \cite{larsen_shalev} and
references
therein.  Whenever $k = o(n)$ the conjecture
is expected to hold, however, in part because the proposed lower bound on mixing
time follows from standard techniques in the field.  We give the second moment
 method proof in Appendix \ref{lower_bound_appendix}.

Since the initial work of Diaconis and Shahshahani,  Conjecture
\ref{main_conjecture} has received quite a bit of attention, see 
 \cite{flatto_odlyzko_wales}, \cite{lulov_thesis}, \cite{roichman_1},
\cite{roichman_2}, 
\cite{roussel_1}, \cite{roussel_2}, \cite{vershik_kerov}  for cases of conjugacy classes with finite numbers of non-fixed points.
A discussion of ongoing work of Schlage-Puchta towards the general conjecture is contained in  \cite{saloff-coste_zuniga}.  See
\cite{diaconis_book} and \cite{saloff-coste_book} for broader
perspective. 

Recently Berestycki, Schramm and Zeitouni
\cite{berestycki}
have established the conjecture for any set of $k$ cycles
with $k$ fixed as $n \to \infty$.  Moreover, they assert that their analysis
will go through to treat the case of any conjugacy class having bounded total
length of non-trivial cycles, and may cover the case when the total cycle
length grows like $o(\sqrt{n})$, although they state that they have not checked carefully the
uniformity in $k$ with respect to $n$.  

The purpose of this article is to prove Conjecture \ref{main_conjecture}
for the random $k$ cycle walk in the full range $k = o(n)$.  
\begin{theorem}\label{main_theorem}
Conjecture \ref{main_conjecture} holds when $C$ is the conjugacy class of
all $k$ cycles with $k$ permitted to be any function of $n$ that satisfies $k =
o(n)$ as $n \to \infty$.  If $k = o\left(\frac{n}{\log n}\right)$ then the
conclusion of Conjecture
\ref{main_conjecture} remains valid when $\varepsilon = \varepsilon(n)$ is any
function satisfying $\varepsilon(n) \log n \to \infty$.
\end{theorem}
\begin{rem}
 When $k = o\left(\frac{n}{\log n}\right)$ the window $\epsilon(n)\log n \to \infty$ is essentially best possible.
\end{rem}

As remarked above, the lower bound was already known.  Prior to
\cite{berestycki}, which is purely combinatorial, all approaches to Conjecture
\ref{main_conjecture} have followed the work of Diaconis and Shahshahani in
passing through bounds for characters on the symmetric group.  We
return to this previous approach;
in particular, we give an asymptotic
 evaluation of character ratios at a $k$ cycle for many
representations, in the range $k = O\left(n^{\frac{1}{2}-\epsilon}\right)$.  The
asymptotic formula
appears to be new when $k > 6$, see \cite{ingram} for the smaller cases.   A
precise statement of our result appears in Section
\ref{S_n_char_theory_section},
after we introduce
the necessary notation. 

The basis for our argument is an old formula of Frobenius \cite{frobenius},
which gives the value of a character of the symmetric group evaluated at a
$k$ cycle as a contour integral of a certain rational function, characterized by
the cycle and the representation. In the special case of a transposition this
formula was already used by Diaconis and Shahshahani. Previous authors in
attempting to extend the
result of \cite{diaconis_shahshahani} had also used Frobenius' formula, but they
had attempted to estimate the sum of residues of the function directly, which
entails  significant difficulties since nearby residues of the function are
unstable once the cycle length $k$ becomes somewhat
large. 
We avoid these difficulties by estimating the integral itself rather
than the residues in most situations.  In doing so, a certain regularity of the
character values becomes evident.  For instance, while nearby residues
appear irregular, by grouping clumps of poles inside a common integral we are
able to show that the `amortized' contribution of any pole is slowly
varying.

It may be initially surprising that our analysis becomes greatly simplified as
$k$ grows.  For instance, once $k$ is larger than a sufficiently large
constant times $\log n$, essentially trivial bounds for the contour integral
suffice, and the greatest part of the analysis goes into showing that our method
can handle the handful of small cases  which had been treated
previously using the character approach.  A similar feature occurs in the
related paper \cite{hough_jiang} where a contour integral is also used
to bound character ratios at rotations on the orthogonal group.  It seems
that the contour method works best when there is significant oscillation in the
sum of residues, but is difficult to use in the cases where the residues are
generally positive.

In Appendix \ref{contour_appendix} we show that
Frobenius' formula has a
natural generalization to other conjugacy classes on the symmetric group, but
with increasing complexity, since one contour integration enters for each
non-trivial cycle.  The analysis here applies more broadly than to
just the class of $k$ cycles, but we have not pushed the method to its limit
because it seems that some new ideas are needed
to obtain the full Conjecture \ref{main_conjecture} when the number of small
cycles in
the cycle decomposition becomes large.  From our point of view, the classes
containing $\geq n^{1-\epsilon}$ 2-cycles would appear to pose the greatest
difficulty. 

We remark that, as typical in the character ratio approach, our upper
bound is a consequence of a corresponding upper bound for the walk in $L^2$, so
that our result gives a broader set of Markov chains on $S_n$ to which one can
apply the comparison techniques of \cite{diaconis_saloff_coste}.

Regarding notation, it will occasionally be convenient to use the
Vinogradov notation $A \ll B$, with the same meaning as $A = O(B)$.  The
implicit constants in notation of both types should be assumed to vary from line
to line.

\subsection*{Acknowledgement} I am grateful to Persi Diaconis and
John Jiang for 
stimulating discussions and to K. Soundararajan for drawing my attention to the
paper \cite{larsen_shalev}.  I thank the referees for correcting a number of minor errors in the original submission.

\section{Character theory and mixing times}
We recall some basic facts regarding the character theory of a finite group.  A good reference for these is
\cite{diaconis_book}.

A conjugation-invariant measure $\mu$ on finite group $G$ is a class
function, which means that it has a `Fourier expansion' expressing $\mu$ as a
linear combination of the irreducible characters $X(G)$ of $G$.  It will
be convenient to normalize the Fourier coefficients by setting 
\[ 
\forall \chi \in X(G), \qquad \hat{\mu}(\chi)= \frac{1}{\chi(1)}\sum_{g \in
G}\mu(g)\overline{\chi(g)}.
\]
By orthogonality of characters, the Fourier expansion takes the form
\begin{equation}\label{fourier_expansion}
 \mu = \frac{1}{|G|}\sum_{\chi \in X(G)}\chi(1)
\hat{\mu}(\chi)
\chi,
\end{equation}
since, writing $C_x$ for the conjugacy class of $x \in G$,
\begin{align*}
\frac{1}{|G|} \sum_{\chi \in X(G)} \chi(1)\hat{\mu}(\chi) \chi(x)  &=
\frac{1}{|G|}
\sum_{\chi \in X(G)} \sum_{g \in G} \mu(g) \overline{\chi}(g) \chi(x)\\  
&= \sum_{g \in C_x} \frac{\mu(g)}{|C_x|}= \mu(x).
\end{align*}
In this setting, the Plancherel identity is
\begin{equation}\label{plancherel}
 \|\mu\|_2^2 = \frac{1}{|G|}\sum_{g \in G} |\mu(g)|^2 =
\frac{1}{|G|^2}\sum_{\chi \in X(G)} \chi(1)^2 |\hat{\mu}(\chi)|^2.
\end{equation}

Note the somewhat non-standard factor of the inverse of the dimension
$\chi(1)^{-1}$ in the Fourier coefficients.  The advantage of this choice is
that the Fourier map satisfies
the
familiar property of carrying convolution to pointwise multiplication: for
conjugation-invariant measures $\mu_1$ and $\mu_2$
\begin{equation}\label{convolution}
 \forall \chi \in X(G), \qquad \widehat{\mu_1 \ast \mu_2}(\chi) =
\hat{\mu}_1(\chi) \cdot \hat{\mu}_2(\chi).
\end{equation}
This is because the regular representation in $L^2(G)$ splits as the direct sum
over irreducible
representations,
\[
 L^2(G) = \bigoplus_{\rho} d_\rho M_\rho,
\]
and convolution by $\mu_i$ acts as a scalar multiple of the identity on each
representation space $M_{\rho}$, the scalar of proportionality being
$\hat{\mu}_i(\chi_\rho)$.

When $G = S_n$ is the symmetric
group on $n$ letters, the irreducible representations are indexed by
partitions of $n$, with the partition $(n)$ corresponding to the trivial
representation, and partition $(1^n)$ corresponding to the sign.  
Given a conjugacy class $C$ on $S_n$ and integer $t
\geq 1$, the measure $\mu_C^{\ast t}$ is conjugation invariant, and supported on
permutations of a fixed sign, odd if both $C$ and $t$ are odd, and otherwise
even.  We let $U_t$ be the uniform measure on permutations of this sign:
\[
 \forall \sigma \in S_n, \qquad U_t(\sigma) = \frac{1 +
\sgn(\sigma)\cdot\sgn(C)^t}{n!}
\]
with Fourier coefficients 
\[
 \hat{U_t}(\chi^n) = 1, \qquad \hat{U_t}(\chi^{1^n}) = \sgn(C)^t, \qquad
\hat{U_t}(\chi^\lambda) = 0, \; \lambda \neq n, 1^n.
\]
The total variation distance between $\mu_C^{\ast t}$ and $U_t$ is equal to
\begin{align*}
 \|\mu_C^{*t} - U_t\|_{T.V.}&:= \sup_{A \subset S_n} \left|\mu_C^{*t}(A) -
U_t(A)\right| \\&=\frac{1}{2}\sum_{\sigma \in S_n} \left|\mu_C^{*t}(\sigma) -
\frac{1
+ \sgn(\sigma)\cdot\sgn(C)^t}{n!}\right|.
\end{align*}
Thus Cauchy-Schwarz, Plancherel, and the convolution identity give the upper
bound 
\begin{equation}\label{upper_bound_lemma}
  \|\mu_C^{*t}-U_t\|_{T.V.} \leq \frac{1}{2} \left(\sum_{\lambda \vdash n,
\lambda
\neq n, 1^n} (\chi^\lambda(1))^2
\left(\frac{\chi^\lambda(C)}{\chi^\lambda(1)}\right)^{2t}\right)^{\frac{1}{2}}.
\end{equation}

The main ingredient in the proof of the upper bound of  Theorem 
\ref{main_theorem} will thus be the following estimate for character
ratios.
\begin{proposition}\label{char_ratio_prop}
 Let $C$ be the class of $k$ cycles
on $S_n$. There exists a constant $\delta > 0$ and a constant $c_1> 0$ such that
uniformly in $n$, partition $\lambda \vdash n$ and all $2 \leq k < \delta n$ 
\[ \left|\frac{\chi^\lambda(C)}{\chi^\lambda(1)}\right|^{\frac{n}{k}(\log n
+ c_1)}
\leq \frac{1}{\chi^\lambda(1)}.\]
\end{proposition}

  For the deduction of Theorem
\ref{main_theorem} we will use the following technical result on the dimensions
of the irreducible representations of $S_n$.
\begin{proposition}\label{dimension_prop}
 For a sufficiently large fixed constant $c_2>0$, 
\[
 \sum_{\lambda \vdash n} (\chi^\lambda(1))^{-\frac{c_2}{\log n}} = O(1),
\]
with the estimate uniform in $n$.
\end{proposition}

\begin{proof}[Deduction of mixing time upper bound of Theorem \ref{main_theorem}]
 We use the fact that, apart from the trivial and sign representations, the lowest dimensional irreducible representation of $S_n$ has dimension $n-1$ (so $S_n$ is essentially quasi-random). Let $c_1, c_2$ as as above, and let $C>0$.  Then for $t > \frac{n}{k_n}(\log n + c_1)(1 + \frac{c_2 + C}{2\log n})$ the application of Cauchy-Schwarz above gives
 \[
  \|\mu_C^{*t}-U_t\|_{T.V.}^2 \leq \frac{1}{4} \sum_{\lambda \vdash n,
\lambda
\neq n, 1^n} (\chi^\lambda(1))^{- \frac{c_2 + C}{\log n}} = O(e^{-C}).
 \]
\end{proof}

The proof of Proposition \ref{char_ratio_prop} will require some more detailed
information regarding the characters of $S_n$ evaluated at a cycle.  We discuss
this in the next section.  Proposition \ref{dimension_prop} is deduced from a
very useful approximate dimension formula of Larsen-Shalev
\cite{larsen_shalev}.  The rather technical proof of this proposition is given
in Appendix \ref{dimension_bound_appendix}.

\section{Character theory of $S_n$}\label{S_n_char_theory_section}

The irreducible representations of $S_n$ are indexed by partitions of $n$. 
Given a partition \[\lambda \vdash n,\qquad \lambda = (\lambda_1 \geq \lambda_2
\geq... \geq \lambda_n \geq 0),\qquad \sum \lambda_i = n,\] the dual partition
$\lambda'$ is found by reflecting the diagram of $\lambda$ along its diagonal.
We write $\chi^\lambda$ for the character of representation $\rho^\lambda$, and
$f^\lambda = \chi^\lambda(1)$ for the dimension.  Let $\mu = \lambda + (n-1,
n-2, ..., 0)$.  Then the dimension is given by (see e.g. \cite{macdonald})
\begin{equation}\label{dimension_formula}
 f^\lambda = \frac{n!}{\mu_1!\mu_2!...\mu_n!}\prod_{i < j} (\mu_i - \mu_j).
\end{equation}

Setting apart those terms that pertain to $\lambda_1$, we find
\[ 
 f^\lambda = \binom{n}{\lambda_1}\prod_{j=2}^n \frac{\lambda_1-\lambda_j +
j-1}{\lambda_1 + j-1} f^{\lambda-\lambda_1}.
\]
The product is bounded by 1, and $\sum_{\lambda \vdash n}
\left(f^\lambda\right)^2 = n!$,
so that we have the bound employed by Diaconis-Shahshahani
\begin{equation}\label{diaconis_shahshahni_dim_bound}
 f^\lambda \leq \binom{n}{\lambda_1}\sqrt{(n-\lambda_1)!}.
\end{equation}

The trivial representation is $\rho^n$ while the sign representation is
$\rho^{1^n}$, both one-dimensional.  $\rho^{n-1,1}$ corresponds to the  standard representation, 
which is the irreducible $(n-1)$-dimensional sub-representation of the representation $\rho$ in $\bR^n$ given by
\[
\rho(\sigma) \cdot e_i = e_{\sigma(i)}, \qquad \sigma \in S_n.
\] This representation and its dual  are the  lowest dimensional non-trivial irreducible representations of $S_n$.

Given the character $\chi^\lambda$, the
character of the dual representation is given by \[\chi^{\lambda'}(\sigma) =
\sgn(\sigma)\cdot\chi^\lambda(\sigma).\]  The characters corresponding to
partitions with long first piece have relatively simple interpretations.  For
instance, if we write $i_1$ for the number of fixed points and $i_2$ for the
number of 2-cycles in permutation $\sigma$ then (\cite{fulton_harris}, ex. 4.15)
\begin{align}\notag
\chi^{n-1,1}(\sigma)& = i_1 -1, \\ \label{small_chars} \chi^{n-2,1,1}(\sigma) &=
\frac{1}{2}(i_1-1)(i_1-2) - i_2,\\ \notag \chi^{n-2,2}(\sigma) &=
\frac{1}{2}(i_1-1)(i_1-2) + i_2-1.
 \end{align}
It is now immediate that
\begin{equation}\label{tensor_decomp}
  (\chi^{n-1,1})^2 = \chi^{n-2,2} + \chi^{n-2, 1, 1} + \chi^{n-1,1} + \chi^{n}.
\end{equation}
Also,
\[
 f^n = 1, \;\; f^{n-1,1} = n-1, \;\; f^{n-2,1,1} = \binom{n-1}{2}, \;\;
f^{n-2,2} = \binom{n-1}{2} - 1.
\]
We  use these formulas in the proof of the lower bound of Conjecture
\ref{main_conjecture}.

 Many properties of partitions are most readily evident in Frobenius
notation, and we will use this notation to state our main technical theorem. 
In Frobenius notation we identify the partion $\lambda \vdash n$
by
drawing the diagonal, say of length $m$, and measuring the legs that extend
horizontally to the right and vertically below the diagonal:
\[ \lambda = (a_1, a_2, ..., a_m| b_1, b_2, ...,  b_m);\]  \[a_ i  = \lambda _i
- i
+
\frac{1}{2},\qquad b_i = \lambda_i' -i + \frac{1}{2}.\]  
 Notice that 
\[
 \sum_{i =1}^m a_i + \sum_{i=1}^m b_i = n.
\]
In Appendix \ref{dimension_bound_appendix} we consider the quantity $\Delta$,
which is the number of boxes contained neither in the square formed by the
diagonal nor in the first row.  
The notation used here is summarized in Figure \ref{frobenius_figure}.  
\begin{figure}                                         
                        
\caption{Partition $\lambda$ in Frobenius
notation.}\label{frobenius_figure} 
\centering
\includegraphics[width = 0.7 \textwidth]{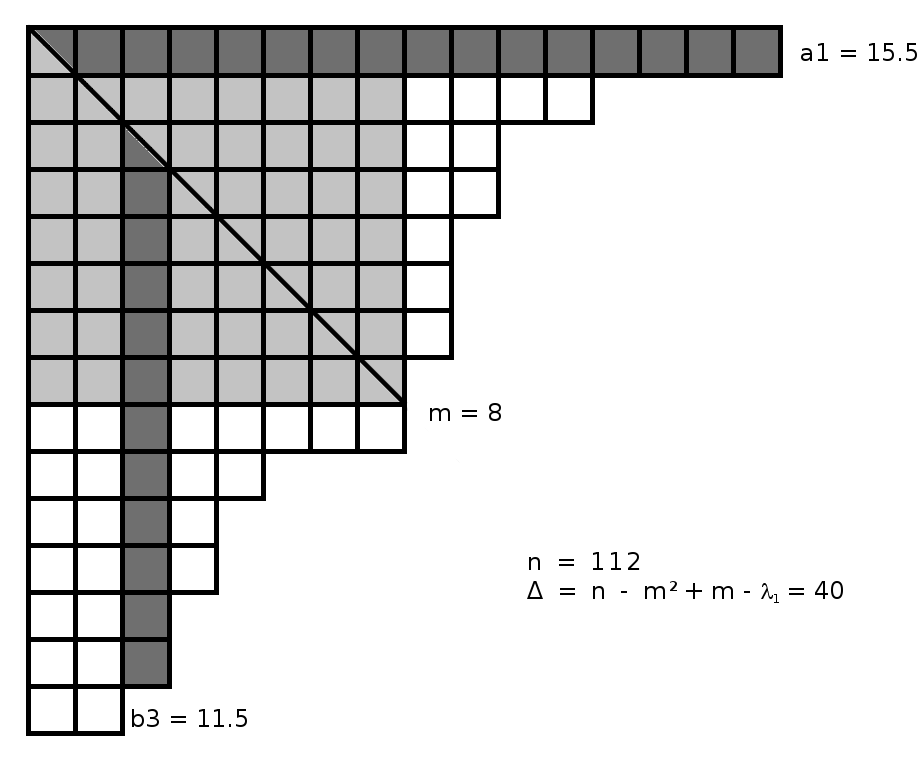}
\end{figure}

  We  now state our asymptotic evaluation of the character ratio. 
\begin{theorem}\label{character_ratio_theorem}
 Let $n$ be large, let $2 \leq k
 \leq n$, and let $C$ be the class of $k$ cycles on $S_n$. Let
$\lambda \vdash
n$ be a
partition of $n$ with Frobenius notation $\lambda = (a_1, ..., a_m|b_1, ...,
b_m)$. 
\begin{enumerate}
\item[a.] (Long first row) Let $0 < \epsilon < \frac{1}{2}$, let $r = n-\lambda_1$ and suppose
that $r + k + 1 < (\frac{1}{2} - \epsilon)n$.  Then 
\begin{align}\label{char_ratio_asym_large}
\frac{\chi^\lambda(C)}{f^\lambda} &= \frac{(a_1 -
\frac{1}{2})^{\underline{k}}}{n^{\underline{k}}} \prod_{j=2}^m \frac{a_1 -a_j -
k}{a_1 - a_j} \prod_{j=1}^m \frac{a_1 + b_j}{a_1 + b_j-k} \\\notag&+
O_\epsilon\left(\exp\left(k \left[ \log \frac{(1 + \epsilon)(k+1+r)}{n-k} +
O_\epsilon\left(r^{-\frac{1}{2}}\right)\right]\right)\right) .\end{align}  If
$r < k$ 
then the
error term is actually 0.
\item[b.] (Large $k$, short first row and column)
Let $\theta > \frac{2}{3}$.  There exists $\epsilon(\theta)>0$ such that,
for all $n$ sufficiently large, for all $k$ with $6\log n \leq k \leq
\epsilon n$ and for all $\lambda \vdash n$ such that $b_1 \leq a_1 \leq
e^{-\theta}n$,
\[
 \left|\frac{\chi^\lambda(C)}{f^\lambda}\right| \leq e^{\frac{-k}{2}}.
\]
\item[c.]  (Asymptotic expansion) Let $0<\epsilon < \frac{1}{2}$ and suppose now that $k <
n^{\frac{1}{2}-\epsilon}$.
We
have the approximate formula
\begin{align}\label{char_ratio_asym_small}
 \frac{\chi^\lambda(C)}{f^\lambda} = &\sum_{a_i > kn^{\frac{1}{2}}}
\frac{a_i^k}{n^k}\left(1 +
O_\epsilon\left(\frac{kn^{\frac{3}{4}+\epsilon}}{a_i}\right)\right)\\&\notag
\;\;+
(-1)^{k-1}\sum_{ b_i > kn^{\frac{1}{2} }} \frac{b_i^k}{n^k}\left(1 +
O_\epsilon\left(\frac{kn^{\frac{3}{4}+\epsilon}}{b_i}\right)\right) \\ \notag
&\;\;+
O_\epsilon\left(n^{\frac{1}{2}}(\log n)^2 \left(\frac{k \log^2
n}{\sqrt{n}}\right)^k\right).
\end{align}
\end{enumerate}
\end{theorem}
This Theorem is the main technical result of the paper.  Actually it gives
 more than we will need, and we apply the detailed statement
of part c. only in the case when the cycle length $k$ is relatively short, $k
\leq 6 \log n$. The cruder bounds of parts a. and b. suffice when the cycle
length $k$ is larger.

\section{Proof of Theorem \ref{character_ratio_theorem}}

When $C$ is the class of $k$ cycles Frobenius \cite{frobenius} (see
\cite{fulton_harris}, p.52
ex. 4.17  b) proved a famous formula that expresses the character ratio of a
given representation at
$C$ as the `residue at $\infty$' of a meromorphic function depending
upon the representation. In Frobenius notation, and using the
falling power  \[z^{\underline{r}} =z(z-1)...(z-r+1),\] the
formula is
\begin{align}\label{contour_short_form}
  \frac{\chi^\lambda(C)}{f^\lambda}
&=\frac{-1}{kn^{\underline{k}}}\frac{1}{2\pi i}\oint F_k^{a,b}(z)
dz, \quad
F_k^{a,b}(z) =  \left(z+
\frac{k-1}{2}\right)^{\underline{k}}F^a_k(z)F^b_k(z)\\ & F^a_k(z) =
\prod_{j=1}^m \frac{z-a_j-\frac{k}{2}}{z-a_j+\frac{k}{2}} , \quad F^b_k(z) =
\prod_{j=1}^m
\frac{z + b_j + \frac{k}{2}}{z+b_j-\frac{k}{2}}\notag
\end{align}
where the integration has winding number 1
around each (finite) pole of the integrand.  [Note that our $(a_i, b_i)$
correspond
to $(b_{m-i+1} + \frac{1}{2}, a_{m-i + 1} + \frac{1}{2})$ of
\cite{fulton_harris}, and replace $y$ there with our $z = y -
\frac{k-1}{2}$.]

Our proof of Theorem \ref{character_ratio_theorem} is an asymptotic
estimation of this integral. We first record several properties of $F_k^{a,b}$
in the following lemma.
\begin{lemma}\label{F_properties}
 Let $\lambda = (a_1, ..., a_m|b_1, ..., b_m)$ be a partition of $n$ in
Frobenius notation.  
\begin{enumerate}
 \item Each of $F_k^{a}(z)$ and $F_k^b(z)$ has at most $\sqrt{n}$ poles.
 \item If $k > n$ then $F_k^{a,b}(z)$ is holomorphic.
 \end{enumerate}
Denote by $\lambda' = \lambda \setminus \lambda_1 = (a_1',
..., a_{m'}'|b_1', ..., b_{m'}')$ the partition of $n-\lambda_1$ found
by deleting the first row of $\lambda$.  
 \begin{enumerate}
 \item[(3)]  We have 
 \[
  F_k^{a,b}(z) \frac{z- a_1 +\frac{k}{2}}{z-a_1 - \frac{k}{2}}
= F_k^{a',b'}(z+1).
 \]
\item[(4)] If $a_1 > n-k - \frac{1}{2}$ and if $n \geq 2k$ then $F_k^{a,b}(z)$
has a
single
simple pole at $z= a_1 - \frac{k}{2}$.
\end{enumerate}

\end{lemma}

\begin{proof}
 The first item follows from the bound for the diagonal $m \leq \sqrt{n}$, since
both $F_k^a$
and $F_k^b$ are products of $m$ terms.
 
 For 2., first observe that $a_i$ and $b_i$ are strictly decreasing so that the
poles of $F_k^a(z)$ are all simple, as are the poles of $F_k^b(z)$.  Since $k >
n$, $a_i + b_j - k < 0$, so $a_i - \frac{k}{2} \neq \frac{k}{2} - b_j$.  Thus
the poles of $F_k^a(z)F_k^b(z)$ are all simple, and are all cancelled by the
factor of $\left(z + \frac{k-1}{2}\right)^{\underline{k}}$.
 
 For 3., observe that deleting the first row of the diagram for
$\lambda$ shifts the diagonal down one square.  Thus, if
$b_m = \frac{1}{2}$ then $m' = m-1$ and 
 \[
 \forall\; 1 \leq i \leq m',\quad a_{i}' = a_{i+1} + 1, \quad b_i' = b_i-1.
 \]
 In this case, accounting for the lost factor from $b_m$,
 \begin{equation}\label{factor_drop}
  F^a_k(z)F^b_k(z)\cdot \frac{z - a_1 + \frac{k}{2}}{z-a_1 - \frac{k}{2}} =
F^{a'}_k(z+1)F^{b'}_k(z+1) \frac{z +\frac{k+1}{2}}{z -\frac{k-1}{2}}.
 \end{equation}
 If instead $b_m > \frac{1}{2}$ then $a_m = \frac{1}{2}$, which implies $m' =
m$, and 
 \[
  \forall\; 1 \leq i \leq m-1, \quad a_{i}' = a_{i+1} + 1,\; a_m' = \frac{1}{2},
\quad \forall\; 1 \leq j \leq m, \quad  b_j' = b_j-1.
 \]
Thus (\ref{factor_drop}) still holds, the new factor accounting for the
introduction of $a_{m}' = \frac{1}{2}$. Thus, in either case,
\begin{align*}
 F_k^{a,b}(z)&\cdot  \frac{z - a_1 + \frac{k}{2}}{z-a_1 - \frac{k}{2}} =
\left(z + \frac{k-1}{2}\right)^{\underline{k}}
F^{a'}_k(z+1)F^{b'}_k(z+1) \frac{z +\frac{k+1}{2}}{z -\frac{k-1}{2}}
\\ &=
\left(z+\frac{k+1}{2}\right)^{\underline{k}} F^{a'}_k(z+1)F^{b'}_k(z+1) \\&=
 F_k^{a',b'}(z+1).
\end{align*}

For 4., the condition $a_1 > n-k - \frac{1}{2}$ and $n
\geq 2k$ implies that $a_1 > k -\frac{1}{2}$.  Thus, for all $i$,  $a_1 + b_i >
k$, so that $a_1 - \frac{k}{2} \neq \frac{k}{2} - b_i$, and $a_1 -\frac{k}{2}$
is a simple pole of $F_k^a(z)F_k^b(z)$. Furthermore, it is not cancelled by
$\left(z+\frac{k-1}{2}\right)^{\underline{k}}$, so that $a_1 - \frac{k}{2}$ is a
simple pole of $F_k^{a,b}(z)$.  It is the only pole, since $F_k^{a',b'}(z)$ has
$n-\lambda_1 < k$, hence is holomorphic.
\end{proof}

Parts a. and b. of Theorem \ref{character_ratio_theorem} are more easily
proven. We
give these proofs immediately, and then prove several lemmas before proving part
c.

\begin{proof}[Proof of Theorem \ref{character_ratio_theorem} part a.] In this part, the first row $\lambda_1$ of partition $\lambda$ is significantly larger than the remainder of the partition, of size $r$.  We extract a residue contribution from the pole corresponding to $\lambda_1$, and bound the remainder of the character ratio by approximating it with a character ratio on $S_r$.

Observe that $\sum_{i>1} a_i + \sum_i b_i  =n-a_1 =  r + \frac{1}{2}$. Recall
that we assume $r + k + 1 \leq \left(\frac{1}{2}-\epsilon\right)n$, and that
 this part of the theorem is the asymptotic
\begin{align*}
\frac{\chi^\lambda(C)}{f^\lambda} &= \frac{(a_1 -
\frac{1}{2})^{\underline{k}}}{n^{\underline{k}}} \prod_{j=2}^m \frac{a_1 -a_j -
k}{a_1 - a_j} \prod_{j=1}^m \frac{a_1 + b_j}{a_1 + b_j-k} \\\notag&+
O_\epsilon\left(\exp\left(k \left[ \log \frac{(1 + \epsilon)(k+1+r)}{n-k} +
O_\epsilon\left(r^{-\frac{1}{2}}\right)\right]\right)\right),\end{align*}
with the error equal to 0 if $r <k$.

The poles of $F_k^{a,b}$ are among the points $a_i - \frac{k}{2}$, $\frac{k}{2}
- b_i$, $i = 1, ..., k$, some of which may be cancelled by the numerator.
The condition $k+ 1 + r \leq \left(\frac{1}{2}-\epsilon\right) n$
guarantees that \[a_1 - \frac{k}{2} = n - r - \frac{k+1}{2} >
\left(\frac{1}{2} + \epsilon\right)n.\] Since $a_i, b_j \leq r+\frac{1}{2}$ for
$i \geq 2$,
$j \geq 1$, it follows that $a_1 - \frac{k}{2}$ is the furthest pole from zero
of the function $F^a_k(z)F^b_k(z)$, and that this is a simple pole of
$F_k^{a,b}(z)$.  

Set $R = (1 + \epsilon)\left(r + \frac{k+1}{2}\right)$ and notice \[R <
\frac{n}{2} < a_1 - \frac{k}{2}, \qquad  R > \max_{i \geq 2, j \geq
1}\left(\left|a_i - \frac{k}{2}\right|, \left|b_j -
\frac{k}{2}\right|\right),\] so that $a_1 - \frac{k}{2}$ is outside the loop
$|z| = R$, while all other poles are inside.  Thus
\begin{equation}\label{error_integral}
 \frac{\chi^\lambda(C)}{f^\lambda} = \frac{-1}{kn^{\underline{k}}}
\Res_{z = a_1 - \frac{k}{2}} F_k^{a,b}(z) + \frac{-1}{k n^{\underline{k}}}
\frac{1}{2\pi i} \int_{|z| = R} F_k^{a,b}(z)dz.
\end{equation}
The residue term 
 is equal to the main term of the theorem, so it remains to bound the integral.

Recall the notation
$(a_1', ..., a_{m'}'|b_1', ..., b_{m'}') = \lambda
\setminus \lambda_1$, and that \[F_k^{a,b}(z) = \frac{z-a_1 -
\frac{k}{2}}{z-a_1+\frac{k}{2}} F_k^{a',b'}(z+1).\]  Thus we express the
integral of
(\ref{error_integral}) as
\begin{equation}\label{residue_extracted}
 \frac{-1}{k n^{\underline{k}}} \frac{1}{2\pi i} \int_{|z| = R} F_k^{a,b}(z) dz=
\frac{-1}{kn^{\underline{k}}} \frac{1}{2\pi i} \int_{|z| = R}
F_k^{a',b'}(z+1) \frac{z- a_1 - \frac{k}{2}}{z-a_1 + \frac{k}{2}} dz.
\end{equation}

When $k > r$, $F_k^{a',b'}(z+1)$ is holomorphic, and so
(\ref{residue_extracted})
is zero, which proves the latter claim of the theorem.  Thus we may now assume
that $r \geq k$.

On the contour $|z| = R$, 
 \[\left|z- a_1 +
\frac{k}{2}\right| \geq a_1 - R - \frac{k}{2} \geq n - (2 + \epsilon)\left(r +
\frac{k+1}{2}\right) \gg_\epsilon n,
\]
 so that $\left|\frac{z- a_1 - \frac{k}{2}}{z-a_1 + \frac{k}{2}}
-1\right| = O_\epsilon \left(\frac{k}{n}\right).$  Thus
(\ref{residue_extracted}) is equal to ($C'$ is the $k$ cycle class on $S_r$)
\begin{equation}\label{contracted_evaluation}
\frac{r^{\underline{k}}}{n^{\underline{k}}} \frac{\chi^{\lambda\setminus
\lambda_1}(C')}{f^{\lambda\setminus \lambda_1}} + 
O_\epsilon\left(\frac{k}{n} \frac{1}{ k n^{\underline{k}}}
\int_{|z| = R}
\left|F_{k}^{a',b'}(z+1)\right|
d|z|\right).
\end{equation}
We bound the first term by $\frac{r^{\underline{k}}}{n^{\underline{k}}}$, since
 all character ratios are bounded by 1.

To bound the integral, note that $F^{a'}_k(z)$ and $F^{b'}_k(z)$
each have at most $\sqrt{r}$ poles.  On the contour $|z| = R$,
\[\left|\frac{z-a_i -\frac{k}{2}}{z-a_i +\frac{k}{2}}\right| \leq 1 +
O_\epsilon\left(\frac{k}{r}\right),\] and the terms in $F^{b'}_k(z)$ are bounded
similarly.  It follows that on $|z| = R$, \[|F^{a'}_k(z+1)F^{b'}_k(z+1)| \leq
\exp\left(O_\epsilon\left(\frac{k}{\sqrt{r}}\right)\right).\]  Meanwhile, also
on $|z| =
R$, 
\[ \left|\left(z + \frac{k+1}{2}\right)^{\underline{k}}\right| \leq \left(|z| +
\frac{k+1}{2}\right)^k \leq 
\left((1 + \epsilon)(k+1+r)\right)^k.\]
Putting these bounds together, we deduce a bound for the second term of
(\ref{contracted_evaluation}) of
\begin{align*} \frac{1}{n\cdot n^{\underline{k}}}&\frac{1}{2\pi}\int_{|z| = R}
|F^{a',b'}_k(z+1)| d|z| \leq\frac{1}{n^{\underline{k}} } \sup_{|z| = R}
\left|F_k^{a',b'}(z+1)\right|
\\ & = O_\epsilon\left(\frac{1}{n^{\underline{k}}} \exp\left(k
\left[\log\left((1 +\epsilon)(r + k + 1)\right) +O_\epsilon
\left(\frac{1}{\sqrt{r}}\right)\right] \right)\right).
\end{align*}
To complete the proof of the theorem, use $n^{\underline{k}} \geq
(n-k)^k$ and note that the term $\frac{r^{\underline{k}}}{n^{\underline{k}}}$
from the character ratio in (\ref{contracted_evaluation}) is trivially absorbed
into this error term.
\end{proof}

\begin{proof}[Proof of Theorem \ref{character_ratio_theorem} part b.]
Choose $\theta_1$ with $\frac{2}{3} < \theta_1 < \theta$.  
Since the poles of $F^{a,b}_k$ are among  $a_i - \frac{k}{2}$ and $
\frac{k}{2}- b_i$, $i = 1, 2, 3, ...$, choosing $\epsilon = \epsilon(\theta)>0
$ sufficiently small, $a_i, b_i <e^{-\theta}n$ and $k \leq \epsilon n$
guarantees that the contour $|z| =R = e^{-\theta_1}
n$
contains all
poles of $F^{a,b}_k(z)$.  Thus
\[
 \frac{\chi^{\lambda}(C)}{f^\lambda} =
\frac{-1}{kn^{\underline{k}}}\frac{1}{2\pi i} \int_{|z| = R}
F^{a,b}_k(z) dz.
\]
Write $ \frac{1}{k n^{\underline{k}}} F^{a,b}_k(z) =\frac{(z +
\frac{k-1}{2})^{\underline{k}}}{k n^{\underline{k}}} F^a_k(z)F^b_k(z)$.  Since
$k
\leq \epsilon n$,  for $|z| = R$, \[\left|\frac{(z +
\frac{k-1}{2})^{\underline{k}}}{n^{\underline{k}}}\right| \leq
\frac{\left(R+\frac{k}{2}\right)^k}{(n-k)^k} \leq \exp(k
(-\theta_1 + o(1))),\] with $o(1)$ indicating a quantity that may be made
arbitrarily small with a sufficiently small choice of $\epsilon$.  Also,
$F^a_k(z)$
and $F^b_k(z)$ are each composed of at
most $ \sqrt{n}$ factors, each of size at most $1 +
O_{\theta,\theta_1,\epsilon}(\frac{k}{n})$. We deduce
that for $|z| = R$,
\[ 
 \left|\frac{1}{k n^{\underline{k}}}
F^{a,b}_k(z)\right| \leq \exp\left(k\left(-\theta_1 +
o_{\theta,\theta_1}(1) \right)\right),
\]
the error term requiring that first $\epsilon$ be small, and then that $n$ be
large.
The length of the contour
is $O(n)=O\left(\exp\left(\frac{k}{6}\right)\right)$.  Since $\theta_1 >
\frac{1}{2} + \frac{1}{6}$, it follows that the bound
$\exp\left(-\frac{k}{2}\right)$ holds for the character ratio for all $n$
sufficiently large. 
\end{proof}

We now turn to part c. of Theorem \ref{character_ratio_theorem}.   We will
again estimate the integral
\begin{equation}\label{contour_short_form_again}
 \frac{\chi^\lambda(C)}{f^\lambda}
=\frac{-1}{kn^{\underline{k}}}\frac{1}{2\pi i}\oint F_k^{a,b}(z)
dz,
\end{equation}
but the idea now
will be to evaluate clumps of
poles together.  With an appropriate choice of contour, at any given point
the poles that are sufficiently far away make an essentially constant
contribution to the integrand, and so may be safely removed, simplifying the
integral.

Denote by 
\[\mathcal{P} = \left\{a_j - \frac{k}{2}: 1\leq j \leq m\right\} \cup
\left\{\frac{k}{2} - b_j:
1 \leq j \leq
m\right\}\] the multi-set (that is, set counted with multiplicities) of poles of
$F_k^{a}(z)$ and $F_k^b(z)$. 
One trivial fact concerning the distribution of poles in $\mathcal{P}$ is the
following bound.
\begin{lemma}\label{trivial_upper_bound}  For any $x > 0$, 
\[
 \#\{p \in \mathcal{P}: |p| > x\} \leq \frac{n + km}{x}.
\]
\end{lemma}
\noindent Indeed, this follows from the fact that 
\[
 \sum_{i=1}^m \left|a_i - \frac{k}{2}\right| + \sum_{i=1}^m\left|\frac{k}{2} -
b_i\right| \leq km + \sum a_i + \sum b_i = km+n.
\]

A useful consequence is that 
for $x$ real, $|x| > k\sqrt{n}$, we can always find a real point $y$ nearby
$x$ with large distance from all poles in $\mathcal{P}$. 
\begin{definition}
 Let $x \in \bR\setminus 0$ and let $L>0$ be a parameter. We say that a real
number $y$ is $L$\emph{-well-spaced} for $x$ with respect to $\mathcal{P}$ if
the bound is satisfied
\begin{equation}\label{well_spaced_condition}
 \sum_{p \in \mathcal{P}: |p-x|< \frac{|x|}{2}}\frac{1}{|p-y|} \leq
\frac{L}{|x|}.
\end{equation}

\end{definition}

The following lemma says that if $|x|$ is sufficiently large, then there
are always many points $y$ nearby $x$ that are well-spaced for $x$ with respect
to $\mathcal{P}$.

\begin{lemma}\label{point_spacing}
Let $n > e^5$ and let $1<k < \sqrt{n}$. 
Let $x$ be real with
$|x| > k \sqrt{n}\log n$.   Then there exists a real
$y$ with $|y-x| < \sqrt{n}$ which is $4 \sqrt{n} \log n$-well-spaced for $x$
with respect to $\mathcal{P}$.
\end{lemma}
\begin{proof}
 Let $I$ be the interval $I = \{y: |y-x| < \sqrt{n}\}$.  By Lemma
\ref{trivial_upper_bound}, the interval $I$ contains at most
\[\frac{n+km}{|x|-\sqrt{n}}
\leq \frac{2n}{|x|-\sqrt{n}} \leq
\frac{2.5\sqrt{n}}{k \log n}\] poles of $\mathcal{P}$.  Deleting the segment of
radius $k$ around each pole in $\mathcal{P}\cap I$ removes a set of total
length at most 
\[
 \frac{5 \sqrt{n}}{ \log n} \leq \sqrt{n}\]
leaving $I'$ with $|I'|\geq
\frac{|I|}{2} = \sqrt{n}$.  Now
\begin{align}\notag
\frac{1}{|I'|} \int_{y \in I'} \sum_{p \in \mathcal{P}, |p-x|< \frac{|x|}{2}}
\frac{dy}{|y-p|}&\leq \frac{1}{|I'|} \sum_{p \in \mathcal{P}, |p-x|<
\frac{|x|}{2}} \int_{u \in I, |u-p|\geq k} \frac{du}{|u-p|}\\ & \leq 
\frac{1}{|I'|} \sum_{p \in \mathcal{P}, |p-x|<
\frac{|x|}{2}}2 \times \int_{k}^{\sqrt{n}} \frac{du}{u} \notag \\&\leq
\frac{\log
n}{|I'|} \left|\left\{p \in \mathcal{P}: |p-x| < \frac{|x|}{2}\right\}\right|.
\label{average_of_p_to_y}
\end{align}
Applying the cardinality bound of Lemma \ref{trivial_upper_bound} a second time (recall
$|I'|\geq \sqrt{n}$) we obtain a bound for (\ref{average_of_p_to_y}) of
\[
 \frac{\log n}{\sqrt{n}} \frac{ (n + km)}{\frac{|x|}{2}} \leq \frac{4 \sqrt{n}
\log
n}{|x|}.
\]
It follows that a typical point $y \in I'$ is $4 \sqrt{n}\log n$-well-spaced
for $x$.  
\end{proof}

Recall that $k \leq n^{\frac{1}{2}-\epsilon}$.  We now assume, as we may, that
$n$ is sufficiently large to guarantee $k \leq \frac{\sqrt{n}}{(\log n)^2}$ and
we choose a sequence of well-spaced points at which we partition our integral.

To find our sequence of points, first set 
\[
 x_0 = \frac{k}{2}\sqrt{n}(\log n)^2
\]
 and
 \[
 x_j = x_0 + 2j \sqrt{n}, \qquad j = 1,2,...
\]
for those $j$ such that $x_j \leq n + 4 \sqrt{n}$.
  In each interval \[[x_j-\sqrt{n}, x_j +
\sqrt{n}), \qquad j = 0, 1, 2, ...\] apply Lemma \ref{point_spacing} to find
$y_j^+$, a 
$4\sqrt{n}\log n$-well-spaced point for $x_j$.
Also find a point $y_j^-$ in each interval $(-x_j-\sqrt{n}, - x_j +
\sqrt{n}]$ which is $4\sqrt{n}\log n$-well-spaced for $-x_j$.  Thus 
we have the
disjoint
intervals
\[I_0 = (y_0^-, y_0^+), \qquad I_j^+ = [y_{j-1}^+,y_j^+), \quad I_j^- =
(y_j^-, y_{j-1}^-], \quad j = 1, 2, 3, ...\] which together cover
\[(-n-\sqrt{n}, n+\sqrt{n}).\] Let $q_j^+$ (resp. $q_j^-, q_0$) be the number of
poles of $\mathcal{P}$ contained in $I_j^+$, (resp. $I_j^-, I_0$).

Around each interval $I_j^+,$ $j \geq 1$ we draw
a
rectangular box 
\[
 \cB_j^+ = \left\{z \in \bC: y_{j-1}^+ \leq \Re(z) \leq y_j^+, \; |\Im(z)|
\leq k q_j^+\right\}
\]
and similarly around $I_j^-$.  If $q_j^\pm = 0$ then the box may be
discarded.  We also draw a large box $\cB_0$ containing the origin, with
endpoints
at $y_0^- \pm i
k\sqrt{n}, y_0^+ \pm i k\sqrt{n}$.  A permissible contour with which to apply
integral formula (\ref{contour_short_form_again})  is given by 
\[
 \partial \cB_0  \cup \bigcup_{j\geq 1} \partial \cB_j^+ \cup \bigcup_{j\geq 1} \partial
\cB_j^-,
\]
each box being positively oriented, see figure \ref{contour_figure}.  We use a
somewhat shortened
version of this contour.

\begin{figure}

\caption{Schematic of the boxes $\cB_j$. Between
successive equally spaced points $x_j, x_{j+1}$ denoted by (+),  point $y_j$ is chosen to
be well-spaced from the set of poles (red).  Box
$\cB_j$ has vertical edges at $\Re(z) = y_{j-1}, y_j$,
and height proportional to the number of poles
contained.}\label{contour_figure} 
\centering
\includegraphics[width = 0.7 \textwidth]{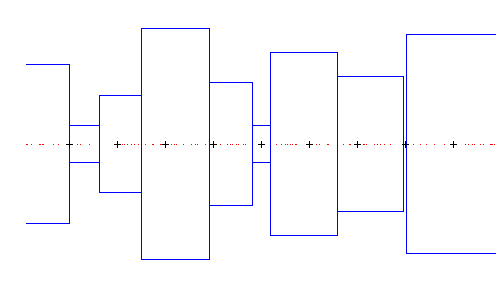}
\end{figure}

In what follows we treat only the integrals around the boxes $\cB_j^+$ and
$\cB_0$,
and we
drop superscripts (so we write e.g. $q_j$ for $q_j^+$, $I_j$ for $I_j^+$ etc).  The argument may be carried out symmetrically for $\cB_j^-$.
The main proposition is as follows.
\begin{proposition}
 Let $j \geq 1$.  There exists a piecewise linear contour $\mathcal{C}_j$,
having total length
 \[
  |\mathcal{C}_j|  \leq 12 k q_j n^{\frac{1}{4}},
 \]
 such that $\mathcal{C}_j$ has winding number 1 around each pole $p \in
\mathcal{P} \cap \cB_j$ and winding number 0 around all poles $p' \in
\mathcal{P}\setminus \cB_j$. For $z\in \mathcal{C}_j$,
$F_k^{a,b}(z)$ satisfies 
\begin{equation}\label{approximate_formula}
 F^{a,b}_k(z) =  \left( x_j +
\frac{k-1}{2}\right)^{\underline{k}} \prod_{p \in \mathcal{P} \cap \cB_j}
\frac{z-p- k}{z-p} +
O\left( k x_j^{k-1} \sqrt{n} \log n\right). 
\end{equation}

\end{proposition}

\begin{proof}
 Recall that $\cB_j$ is a box containing $q_j$ poles and surrounding interval
$I_j = [y_{j-1}, y_j)$, with corners at
\[
 y_{j-1} \pm ikq_j, \qquad y_j \pm ikq_j.
\]
Recall also that $y_j \in [x_j - \sqrt{n}, x_j+\sqrt{n})$ and
$x_j -
x_{j-1} = 2\sqrt{n}$, so that $|y_{j} - y_{j-1}| \leq 4 \sqrt{n}$.

\begin{figure}
\caption{Schematic of contour $\cC_j$.  Poles in red, box $\cB_j$ in blue, $\cC_j$ in green. }\label{C_contour_figure} 
\centering
\includegraphics[width =  \textwidth]{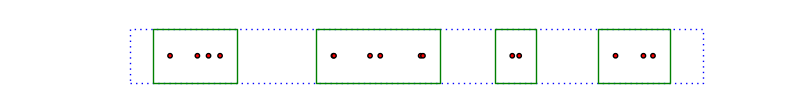}
\end{figure}

To form $\mathcal{C}_j$ we shorten the contour $\partial \cB_j$.  At each pole
$p
\in \cB_j$ consider the segment
\[S_p = [p-kq_j, p+kq_j].\]  If there exists  a subinterval $J$ of interval
$I_j$
that is not covered by $\cup_{p \in \mathcal{P}\cap \cB_j}S_p$  then such $J$
may
be discarded from the box $\cB_j$ by drawing vertical segments at the endpoints
of
$J$ and deleting the parts of $\cB_j$ vertically above and below $J$.
Deleting subintervals in this way if it reduces the total perimeter,
or allowing them to remain when the perimeter is increased, we arrive at a
contour $\mathcal{C}_j$ which is the union of some rectangles, and has total
length at most 
\[ |\mathcal{C}_j| \leq \min(8 \sqrt{n} + 4 k q_j, 8kq_j^2),\]
the first term bounding the result of doing nothing, and the second
bounding the length of a contour with an individual box around every pole. 
Using $\min(a,b) \leq
\sqrt{ab}$,
we have the bound
\[
 |\mathcal{C}_j| \leq 4k q_j + \min(8 \sqrt{n}, 8kq_j^2) \leq 4kq_j + 8 k q_j
n^{\frac{1}{4}}  \leq 12 k q_j n^{\frac{1}{4}}.
\]

We now prove the  formula (\ref{approximate_formula}).  Observe that
by Lemma \ref{trivial_upper_bound}, the number of poles $q_j$ in $\mathcal{P}
\cap \cB_j$ satisfies 
 $q_j \ll \frac{n}{x_j} \ll \frac{\sqrt{n}}{k}$, so that each point of
$\mathcal{C}_j$ has
distance $O(\sqrt{n})$ from $x_j$.  Thus
\[
\forall z \in \mathcal{C}_j, \qquad\left(z + \frac{k-1}{2}\right)^{\underline{k}} = \left( x_j +
\frac{k-1}{2}\right)^{\underline{k}}\left(1 +
O\left(\frac{k\sqrt{n}}{x_j}\right)\right)
.\]
In particular, it follows that
\begin{equation}\label{first_approximation_j}
 \forall z \in \mathcal{C}_j, \qquad F_k^{a,b}(z) = \left(1 + O \left(\frac{k
\sqrt{n}}{x_j}\right)\right) \left( x_j +
\frac{k-1}{2}\right)^{\underline{k}}F_k^a(z)F_k^b(z).
\end{equation}

We  prove 
\begin{equation}\label{distant_poles_bound}
 \forall z \in \mathcal{C}_j, \qquad F_k^a(z)F_k^b(z) = \left( 1 +
O\left(\frac{k\sqrt{n} \log n}{x_j}\right)\right) \cdot \prod_{p \in \mathcal{P}
\cap \cB_j}\frac{z-p-k}{z-p},
\end{equation}
and
\begin{equation}\label{near_poles_bound}
 \forall z \in \mathcal{C}_j, \qquad \prod_{p \in \mathcal{P}
\cap \cB_j}\frac{z-p-k}{z-p} = O(1).
\end{equation}
Inserted in (\ref{first_approximation_j}) these combine to prove the
 formula (\ref{approximate_formula}), (use $ \left( x_j +
\frac{k-1}{2}\right)^{\underline{k}} \leq x_j^k$).

We first consider (\ref{distant_poles_bound}).  This 
estimate is equivalent to \[
 \forall z \in \mathcal{C}_j, \qquad \prod_{p \in \mathcal{P} \setminus \cB_j}
 \frac{z - p \mp k}{z-p} = 1 + O\left(\frac{k\sqrt{n} \log n}{x_j}\right),
\]
the sign determined according as $p$ is a pole from $F_k^a$ or $F_k^b$.
Say a pole $p \in \mathcal{P} \setminus \cB_j$ is `good' if
$|p-x_j| > \frac{x_j}{2}$ and `bad' otherwise.  For $z \in \mathcal{C}_j$,
\[|z-x_j| = O(\sqrt{n}) < \frac{x_j}{3}\] if $n$ is sufficiently large, which
implies that for good $p$, $|z-p| \geq \frac{x_j}{6}$.   Since the total number
of poles is $O(\sqrt{n})$, the product over good poles evidently satisfies the
bound
\[
 \prod_{p \text{ good}} \left(1 \mp \frac{k}{z-p}\right) = 1 +
O\left(\frac{k\sqrt{n}}{x_j}\right).
\]

Since it is not contained in the box $\cB_j$, a bad pole $p$ necessarily lives
in
the complement $\bR
\setminus [y_{j-1},y_j]$, and either satisfies $p > y_j$ and $|p- x_j|\leq
\frac{x_j}{2}$ or
else $p < y_{j-1}$, which entails $\frac{x_j}{2} \leq p \leq x_j$.  Since
$\frac{x_{j-1}}{2} < \frac{x_j}{2}$, and $\frac{3 x_{j-1}}{2} > x_j$ (at least
if $n$ is sufficiently large), it follows that in the second case, $|p -
x_{j-1}| \leq \frac{x_{j-1}}{2}$. Thus, for $z \in
\mathcal{C}_j$, the fact that $y_{j-1} \leq \Re(z) \leq y_j$ implies
\[
 \sum_{p \text{ bad},\; p > y_j} \frac{1}{|z-p|} \leq \sum_{\substack{p > y_j\\
|p-x_j| < \frac{x_j}{2}}} \frac{1}{p - y_j} \leq \frac{4 \sqrt{n} \log n}{x_j}
\]
and 
\[
 \sum_{p \text{ bad},\; p< y_{j-1}} \frac{1}{|z-p|} \leq
\sum_{\substack{p<y_{j-1}\\ |p -x_{j-1}| < \frac{x_{j-1}}{2} }} \frac{1}{y_{j-1}
- p} \leq \frac{4 \sqrt{n}\log n}{x_{j-1}} \leq
\frac{8 \sqrt{n} \log n}{x_j},
\]
by invoking the $4 \sqrt{n}\log n$ well-spaced property of $y_{j-1}$ and
$y_j$. 
It follows 
\[
 \prod_{p \text{ bad}} \left(\frac{z-p \mp k}{z-p}\right) = 1 +
O\left(\frac{k\sqrt{n}\log n}{x_j}\right),
\]
proving (\ref{distant_poles_bound}).

To prove (\ref{near_poles_bound}) we consider two cases.  When either $\Re(z) =
y_{j-1}$ or $\Re(z) = y_j$, observe that for each $p \in \mathcal{P} \cap
\cB_j$,
\[|p - x_j| \leq 2 \sqrt{n} \leq \frac{x_j}{2}.\]  Also, since $x_j-x_{j-1} =
2\sqrt{n}$, \[|p - x_{j-1}| \leq 4 \sqrt{n} \leq \frac{x_{j-1}}{2} \qquad (n
\text{ large}).\]   Thus, invoking the well-spaced property of $y_{j-1}$ and
$y_j$,
\[
 \sum_{p \in \mathcal{P} \cap \cB_j} \frac{1}{|p-y_{j-1}|} \leq \frac{4
\sqrt{n}\log n}{x_{j-1}}, \qquad \sum_{p \in \mathcal{P} \cap \cB_j}
\frac{1}{|p - y_j|} \leq \frac{4 \sqrt{n} \log n}{x_j}
\]
so that if $\Re(z) = y_{j-1}$,
\[
 \prod_{p \in \mathcal{P} \cap \cB_j} \left(1 - \frac{k}{z-p}\right)
\leq \exp\left(O \left(\frac{k \sqrt{n}\log n}{x_{j-1}}\right)\right) = O(1),
\]
and similarly when $\Re(z) = y_j$.

When $\Re(z) \not \in \{y_{j-1}, y_j\}$ then $|z-p| \geq k q_j$ for each $p \in
\mathcal{P}\cap \cB_j$, so that in this case the bound
\[
 \prod_{p \in \mathcal{P} \cap \cB_j} \left(1 - \frac{k}{z-p}\right) = O(1)
\]
is immediate from $|\mathcal{P} \cap \cB_j| = q_j$.  Thus
(\ref{near_poles_bound}) holds in either case.

\end{proof}

We now bound integration around the box $\cB_0$. 
\begin{lemma}\label{cB_0_bound}  Assume $n > e^5$.
We have
\[
 \frac{-1}{k n^{\underline{k}}} \frac{1}{2\pi i} \oint_{\partial \cB_0}
F^{a,b}_k(z) dz = O\left(\sqrt{n} (\log
n)^2 \left(\frac{k\log^2
n}{n^{\frac{1}{2}}}\right)^k \right).
\]

\end{lemma}
\begin{proof}
Since $\cB_0$ is a box with corners at $y_0^{\pm} \pm ik\sqrt{n}$, the integral
has length $O(k\sqrt{n}(\log n)^2)$, so it will suffice to prove
\[
\forall z \in \partial \cB_0, \qquad \left|\frac{F_k^{a,b}(z)}{k
n^{\underline{k}}} \right| =
O\left(\frac{1}{k}\left(\frac{k\log^2
n}{n^{\frac{1}{2}}}\right)^k\right).
\]

Recall that 
\[
 F_{k}^{a,b}(z) =  \left(z +
\frac{k-1}{2}\right)^{\underline{k}} F_k^a(z)F_k^b(z).
\]
Recall also that $|y_0^{\pm} \mp x_0| \leq \sqrt{n}$ and $x_0 =
\frac{k}{2}\sqrt{n}(\log n)^2$.  Thus, on $\partial \cB_0$ we have $|\Re z| \leq
\frac{k}{2}\sqrt{n} (\log n)^2 + \sqrt{n}$.  Hence, using $k < \sqrt{n}$,
\begin{align*}\left(z + \frac{k-1}{2}\right)^{\underline{k}} &\leq
\left(|\Re(z)| + |\Im(z)| + \frac{k-1}{2} \right)^k  \\&\leq
\left(\frac{k}{2}\sqrt{n} (\log n)^2 + 2\sqrt{n} + k \sqrt{n}\right)^k \\&\leq
(k \sqrt{n}(\log n)^2)^k,\end{align*} the last bound requiring $k+2 \leq
\frac{k}{2}(\log n)^2$, which plainly holds for $n > e^5$. 
Since $n^{\underline{k}} \gg n^k$ when $k < \sqrt{n}$, we obtain the required
bound for $\sup_{z \in \partial \cB_0}
\left|\frac{F_{k}^{a,b}(z)}{kn^{\underline{k}}}\right|$ by checking that
\[\sup_{z \in
\partial \cB_0}|F_k^a(z)F_k^b(z)|
= O(1).\]

To do so, for each $p
\in \mathcal{P}$ write the corresponding
factor of $F^a_k(z)$ or $F^b_k(z)$ as 
\[
\frac{ z-p\pm k}{z-p} = 1 \pm \frac{k}{z-p}.
\]
For a given $z \in \partial \cB_0$, say the pole $p \in \mathcal{P}$ is `good'
for $z$ if $|z-p| \geq k \sqrt{n}$.  Since the total number of poles in
$\mathcal{P}$ is $O(\sqrt{n})$, the part of the product in $F^a_k(z)F^b_k(z)$
contributed by good poles is bounded by $O(1)$, so we may consider only the
bad poles. 

Bad poles exist only if $z$ is on one of the two vertical sides
of the box, that is $\Re(z) = y_0^{\pm}$.  Notice that, e.g.
\[
 |p-y_0^+| < k\sqrt{n} \qquad \Rightarrow \qquad |p - x_0| < (k+1)\sqrt{n} <
\frac{k}{4}\sqrt{n}(\log n)^2 = \frac{x_0}{2}
\]
(similarly if $p$ is near $y_0^-$), and so bad poles satisfy
$\min(|p + x_0|, |p - x_0|) <
\frac{x_0}{2}$. 

By the well-spaced property of $y_0^\pm$, when $\Re(z) = y_0^\pm$ we have
\begin{align*}
 \sum_{p \text{ bad}} \frac{1}{|z - p|} &\leq \sum_{p \text{ bad}, p>0}
\frac{1}{|y_0^+ - p|} + \sum_{p \text{ bad}, p < 0} \frac{1}{|y_0^- - p|}\\&
\leq 
\frac{8 \sqrt{n} \log n}{x_0}  \leq \frac{16}{k \log n}.
\end{align*}
It follows that 
\[
 \prod_{p \text{ bad}}\left(1 +  \frac{1}{|z-p|}\right) = O(1).
\]

\end{proof}

We  now have all of the estimates that we need to complete the proof of
Theorem \ref{character_ratio_theorem}.  We  use the following 
integral formula.

\begin{lemma}\label{exact_integral}
 Let $p_1, p_2, ..., p_s \in \bC$ be a sequence of poles and let $d$ be an
arbitrary
complex number.  Then
\[
  I(p_1, ..., p_s, d) = \frac{1}{2\pi i}\oint \prod_{j=1}^s \frac{z-p_j
+d}{z-p_j} = sd
\]
where the contour is such that it has winding number one about each
pole.
\end{lemma}

\begin{proof}
 The integral is independent of the appropriate contour, so take the contour to
be a large loop containing the poles. We may write the integrand as
$\prod_{j=1}^s \left(1 + \frac{d}{z-p_j}\right).$ Differentiating the integrand
with respect to $p_j$ results in a rational function having total degree $-2$. 
It follows that $\frac{\partial}{\partial p_j} I(p_1, ..., p_s, d) = 0$ for each
$j$, since the integral can be made made arbitrarily small by taking the contour
to be sufficiently large.  Taking $p_j = 0$ for each $j$, we find that the
residue at 0 is $sd$.
\end{proof}

\begin{proof}[Proof of Theorem \ref{character_ratio_theorem} part c.]
Write
\[
 \frac{\chi^\lambda(C)}{f^\lambda} = \frac{-1}{kn^{\underline{k}}} \frac{1}{2\pi
i}\left\{
\oint_{\partial \cB_0} + \sum_{j \geq 1} \oint_{\mathcal{C}_j^+} + \sum_{j \geq
1} \oint_{\mathcal{C}_j^-}\right\} F_k^{a,b}(z) dz.
\]

The contribution from $\partial \cB_0$ is bounded in Lemma \ref{cB_0_bound}
by
\[\frac{-1}{kn^{\underline{k}}} \frac{1}{2\pi i} \oint_{\partial \cB_0}
F^{a,b}_k(z) dz = O\left(\sqrt{n} (\log
n)^2 \left(\frac{k\log^2
n}{n^{\frac{1}{2}}}\right)^k \right).
\]
which may be absorbed into the error term of the theorem.

For each $j \geq 1$, write the contribution from  the $q_j^+$ poles inside $\mathcal{C}_j^+$ as 
\[
 \frac{1}{2\pi i} \int_{\mathcal{C}_j^+} \left[  \frac{-( x_j +
\frac{k-1}{2})^{\underline{k}}}{k
n^{\underline{k}}} \prod_{p \in \mathcal{P} \cap \cB_j} \frac{z-p- k}{z-p} +
O\left( \left(\frac{x_j}{n}\right)^{k-1} \frac{\log n}{\sqrt{n}}\right)\right]
dz.
\]
For the main term of this integral,   Lemma \ref{exact_integral} gives the 
evaluation 
\begin{align*}
 \frac{1}{2\pi i} \int_{\mathcal{C}_j^+} \frac{-( x_j +
\frac{k-1}{2})^{\underline{k}}}{k
n^{\underline{k}}}& \prod_{p \in \mathcal{P} \cap \cB_j} \frac{z-p- k}{z-p} dz 
=
q_j \frac{\left(x_j + \frac{k-1}{2}\right)^{\underline{k}}}{n^{\underline{k}}}
\\&= \sum_{a_i: a_i - \frac{k}{2} \in \cB_j^+} \left(\frac{a_i}{n}\right)^k
\left(1 +
O\left(\frac{k\sqrt{n}}{a_i}\right)\right),
\end{align*}
the error resulting since each of the $k$ terms of
$\left(x+\frac{k-1}{2}\right)^{\underline{k}}$ is equal to $a_i + O(\sqrt{n})$
-- we use that $a_i = \Omega(k\sqrt{n})$ here.
Had we considered $\mathcal{C}_j^-$ rather than $\mathcal{C}_j^+$ this main term
would involve a
sum over the $b_i$, with an appropriate sign factor.

The integral of the error term  over
$\mathcal{C}_j$ is bounded trivially by using $|\mathcal{C}_j| = O(k q_j
n^{\frac{1}{4}})$, which gives
\[
 O\left(k q_j n^{\frac{1}{4}} \cdot  \left(\frac{x_j}{n}\right)^{k-1} \frac{\log
n}{\sqrt{n}} \right).
\] The ratio of the
error from $\mathcal{C}_j$ 
to the corresponding main term is
\[
 O\left(\frac{k n^{\frac{3}{4}} \log n}{x_j}\right),
\]
and since each $a_i$ in the sum above is within constants of $x_j$, this
proves the Theorem.

Note that in the
claimed formula of the Theorem, the sums over $a_i$ and $b_j$ contain terms
for which $k\sqrt{n} < a_i  \leq y_0^+ + \frac{k}{2}$
and $k \sqrt{n} < b_j \leq -y_0^- + \frac{k}{2}$, which do not appear in
our
evaluation. However, only a bound, rather than an asymptotic, is claimed in this
range, so that the theorem is valid with the extra terms discarded.
\end{proof}

\section{Proof of upper bound for Theorem
\ref{main_theorem}}\label{upper_bound_section}
We now prove the upper bound on mixing time from Theorem \ref{main_theorem} by
proving Proposition \ref{char_ratio_prop}.  Recall that this proposition is a
bound for character ratios at class $C$ of $k$ cycles,
\begin{equation}\label{target_char_ratio_bound}
 \left|\frac{\chi^\lambda(C)}{f^\lambda}\right|^{\frac{n}{k}(\log n + c)}
\leq \frac{1}{f^\lambda},
\end{equation}
uniformly for $k$ less than a fixed constant times $ n$ and for all $\lambda
\vdash n$.

Before proving the estimate, we collect together several 
observations.  First note that the character ratio bound is trivial for the
one-dimensional representations $\lambda = (n), (1^n)$.  Also, exchanging
$\lambda$ with its dual $\lambda'$ leaves the dimension $f^\lambda$ unchanged,
and at most changes the sign of $\chi^\lambda(C)$, so we will assume
without loss of generality that $\lambda$ has $a_1 \geq b_1$.  We set $r =
n-\lambda_1 = n-a_1 - \frac{1}{2}$.

\begin{lemma}[Criterion lemma]\label{the_target_lemma} To prove the character ratio bound (\ref{target_char_ratio_bound}), it suffices to prove that for all $n$ larger than the
maximum of $k$ and a fixed constant,  and for
a
sufficiently large $c > 0$, that for all non-trivial $\lambda$ with $a_1 \geq
b_1$,
 \[
  \log \left|\frac{\chi^\lambda(C)}{f^\lambda}\right| \leq \max \left(
\frac{-kr}{n} + \frac{kr \log r}{2n \log n} + \frac{ckr}{n\log n},
\frac{-k}{2} + \frac{ck}{\log n}\right).
 \]
\end{lemma}

\begin{proof}
 By the bound
(\ref{diaconis_shahshahni_dim_bound}) for the dimension, we have (use $\log r!
\geq r \log \frac{r}{e}$)
\[
 f^\lambda \leq \binom{n}{r} \sqrt{r!} \leq \frac{n^r}{\sqrt{r!}}\leq 
\exp\left(r \log n -
\frac{1}{2}\left(r \log \frac{r}{e}\right)\right)
\]
so that a bound of 
\begin{equation*}\label{long_target}
 \log \left|\frac{\chi^\lambda(C)}{f^\lambda}\right| \leq \frac{-kr}{n} +
\frac{kr \log r}{2 n \log n} + \frac{c kr}{n \log n}
\end{equation*}
is sufficient.

On the other hand, for all $\lambda$, $f^\lambda \leq \sqrt{n!}$, so that a
bound of 
\[
 \log \left|\frac{\chi^\lambda(C)}{f^\lambda}\right| \leq \frac{-k}{2} +
\frac{ck}{\log n}
\]
also suffices.
\end{proof}

Our proof splits into two cases depending upon whether $k \geq 6\log n$.  The
essential tool in both cases will be the evaluation of the character ratio from
Theorem
\ref{character_ratio_theorem}.  
For large $k$ we will use only parts a. and b.
of that theorem.  Recall that part a. had a main term equal to
\[
 \MT := \frac{(a_1 -
\frac{1}{2})^{\underline{k}}}{n^{\underline{k}}} \prod_{j=2}^m \frac{a_1 -a_j -
k}{a_1 - a_j} \prod_{j=1}^m \frac{a_1 + b_j}{a_1 + b_j-k}.
\]
\begin{lemma}\label{MT_lemma} Assume $k + r+1 <
\frac{n}{2}$. We have the bound
$
 \MT \leq \exp\left(\frac{-kr}{n}\right).
$
\end{lemma}

\begin{proof}
Recall $a_1 = n-r-\frac{1}{2}$. We estimate
 \begin{align*}
\MT&\leq \prod_{i=1}^k \left(\frac{n-r-i}{n+1-i}\right)\prod_{j=2}^m
\left(1 -
\frac{k}{n-r}\right)\prod_{j=1}^m\left(1 + \frac{k}{n-k-r}\right) 
\\& \leq \frac{n-r}{n-k-r} \prod_{i=1}^k \frac{n-r-i}{n-i+1} \\&= \prod_{i=1}^k
\frac{n-r-i+1}{n-i+1} \leq \left(1-\frac{r}{n}\right)^k \leq
\exp\left(\frac{-kr}{n}\right).
\end{align*}
\end{proof}

\begin{proof}[Proof of Proposition \ref{char_ratio_prop} when $6 \log n \leq k
\leq \delta n$]

When $r > 0.49 n$ we have $a_1 < 0.51 n$.  Let $\theta = 0.67$, and note that
$e^{-\theta} > 0.511$.  Thus $a_1  < e^{-\theta}n$ so that part
b of Theorem \ref{character_ratio_theorem}  guarantees that there exists
$\delta =\epsilon(0.67)>0$, such that, for $n$ larger than a fixed constant, for
all $6 \log n \leq k \leq
\delta n$,
$\left|\frac{\chi^\lambda(C)}{f^\lambda}\right| \leq e^{\frac{-k}{2}}.$ Thus the
second condition of Lemma \ref{the_target_lemma} is satisfied.

So we may suppose that $r \leq 0.49 n$ and appeal to part a. of Theorem
\ref{character_ratio_theorem}. Suppose that $\delta$ is sufficiently small so
that   that $r + k + 1 <
\frac{n}{2}$. If $r < k$ then the error term of part a. of Theorem
\ref{character_ratio_theorem} is
zero so that the previous lemma implies 
\[
 \left|\frac{\chi^\lambda(C)}{f^\lambda}\right| = \MT \leq
\exp\left(\frac{-kr}{n}\right).
\]
Thus the first criterion of Lemma \ref{the_target_lemma} is satisfied with $c =
0$.

Assume now that $r \geq k$.  Let now $\delta$  sufficiently small
so that we may choose $\epsilon = \frac{1}{200}$ in part a, that is, $r + k +
1 \leq \left(\frac{1}{2} - \frac{1}{200}\right)n$, and also assume that
$\frac{(1 + \epsilon)(k + r+1)}{n-k} < 0.5 - \eta$ for some fixed $\eta
> 0$.  
Then  part a. gives $\left|\frac{\chi^\lambda(C)}{f^\lambda}\right| \leq \MT +
\ET$ with 
\[
 \ET  \ll \exp \left(k \left[ \log \frac{(1 + \epsilon) (k+r+1)}{n-k} +
O_\epsilon\left(r^{\frac{-1}{2}}\right)\right]\right) \leq 2^{-k} \; (n
\text{ large}).
\]
Since $r < \frac{n}{2}$ we deduce that
\begin{align*}
 \MT + \ET &\leq \exp\left(\frac{-kr}{n}\right) + \exp\left(-k \log 2\right)
\\&\leq \exp\left(\frac{-kr}{n}\right) \left(1 + \exp\left(-k\left(\log 2 -
\frac{1}{2}\right)\right)\right).
\end{align*}
Since $k \geq 6 \log n$ and $6(\log 2 - .5) >1.15$ we deduce
\[
\log \left|\frac{\chi^\lambda(C)}{f^\lambda}\right| \leq \log(\MT + \ET) \leq
\frac{-kr}{n} + O(n^{-1.15}),
\]
so that the first criteria of Lemma \ref{the_target_lemma} is satisfied.

\end{proof}

When $k \leq 6\log n$ we make essential use of the asymptotic evaluation of the
character ratio proved in part c. of Theorem \ref{character_ratio_theorem}.  The
next lemma shows that we may restrict attention to only the main term of that
evaluation.

\begin{lemma}[Small $k$ criterion lemma]\label{condition_lemma}
 Let $2 \leq k \leq 6 \log n$.  We have the bound
 \begin{align*}
  &\left|\frac{\chi^\lambda(C)}{f^\lambda}\right|\leq\\ &\quad \left(1 + O
\left(\frac{\log
n}{n^{\frac{1}{4}}}\right)\right)\left[\sum_{a_i > k
n^{\frac{1}{2}}} \frac{a_i^k}{n^k} + \sum_{b_i > k n^{\frac{1}{2}}}
\frac{b_i^k}{n^k}\right] +
O\left(\frac{e^{-k}(\log n)^4}{n^{\frac{1}{4}}}\right).
 \end{align*}
In particular, if $r > n^{\frac{5}{6}}$ and $n$ is sufficiently large, and if
\begin{align*}
& \log \left(\sum_{a_i > k n^{\frac{1}{2}}} \frac{a_i^k}{n^k} + \sum_{b_i > k
n^{\frac{1}{2}}} \frac{b_i^k}{n^k}\right)\\ & \qquad\qquad\qquad\qquad\leq
\max\left(
\frac{-kr}{n} + \frac{kr \log r}{2n \log n} + \frac{ckr}{n \log n}, \frac{-k}{2}
+ \frac{ck}{\log n}\right)
\end{align*}
then after changing constants, the same estimate holds for
$\log \left| \frac{\chi^\lambda(C)}{f^\lambda}\right|$, so that the condition of
Lemma \ref{the_target_lemma} is satisfied.
\end{lemma}
\begin{proof}
 To deduce  the second statement from the first, note that
\[\log\left(1 + O\left(\frac{\log n}{n^{\frac{1}{4}}}\right)\right) =
O\left(\frac{\log n}{n^{\frac{1}{4}}}\right),\] which, for $r >
n^{\frac{5}{6}}$ may be absorbed into the RHS of the second statement by
increasing the value of $c$.   Regarding
the error of $O\left(\frac{e^{-k}(\log n)^4}{n^{\frac{1}{4}}}\right)$, it is no
loss of generality to assume that \[\sum\frac{a_i^k}{n^k} +
\sum\frac{b_i^k}{n^k} \geq e^{\frac{-k}{2}},\] so that this error term has
relative size $1 + O\left(\frac{(\log
n)^4e^{-\frac{k}{2}}}{n^{\frac{1}{4}}}\right)$.
Again, the logarithm  of this error may be absorbed by increasing $c$.

To prove the first statement, part c. of Theorem \ref{character_ratio_theorem}
gives
\begin{align*}
 \left|\frac{\chi^\lambda(C)}{f^\lambda}\right| \leq& \sum_{a_i >
kn^{\frac{1}{2}}} \frac{a_i^k}{n^k}\left(1 +
O\left(\frac{kn^{\frac{3}{4}}}{a_i}\right)\right) + \sum_{b_i > k
n^{\frac{1}{2}}} \frac{b_i^k}{n^k}\left(1 + O
\left(\frac{kn^{\frac{3}{4}}}{b_i}\right)\right)\\ & \quad +
O\left(n^{\frac{1}{2}}(\log n)^2\left(\frac{k \log n}{\sqrt{n}}\right)^k\right).
\end{align*}
For $2 \leq k \leq 6 \log n$, the last error term is plainly
$O\left(\frac{e^{-k}(\log n)^4}{\sqrt{n}}\right)$.  Split the sum over $a_i$
according to
$a_i \geq \frac{n}{e^2}$.  Thus the sum over $a_i$ is bounded by
\begin{align*}
& \left(1 + O\left(\frac{\log n}{n^{\frac{1}{4}}}\right)\right) \sum_{a_i \geq
\frac{n}{e^2}} \frac{a_i^k}{n^k} + \sum_{kn^{\frac{1}{2}} < a_i <
\frac{n}{e^2}} \frac{ a_i^k}{n^k} \\ &\qquad \qquad\qquad\qquad+ O\left(\frac{
e^{-k}}{n^{\frac{1}{4}}}
\sum_{kn^{\frac{1}{2}} < a_i < \frac{n}{e^2}} \frac{k(e
a_i)^{k-1}}{n^{k-1}}\right).
\end{align*}
In the last sum, $k\left(\frac{ e a_i}{n}\right)^{k-1}$ is minimized at $k=2$. 
Thus
\[
 \sum_{kn^{\frac{1}{2}} < a_i < \frac{n}{e^2}} \frac{k(e
a_i)^{k-1}}{n^{k-1}} \leq 2e \sum_{a_i} \frac{a_i}{n} = O(1).
\]
Handling the sum over $b_i$ in the same way proves the lemma.
\end{proof}

Recall that we require $a_1 \geq b_1$ and that $a_1 = n-r-\frac{1}{2}$. 
Thinking of $r$ as fixed, set $\delta := \frac{r}{n}$.  Since $\sum a_i + \sum
b_i = n$, an upper bound for $\sum \frac{a_i^k}{n^k} + \sum \frac{b_i^k}{n^k}$
is given by the solution to the following optimization problem.

Let $x_1, x_2, x_3, ...$ be real variables and let $k \geq 2$.
\begin{align*}
 \text{maximize}:& \qquad \sum x_i^k \\
 \text{subject to}:& \qquad 0 \leq x_i \leq 1-\delta, \; \sum x_i = 1 .
\end{align*}

Let $\ell = \left \lfloor (1-\delta)^{-1} \right \rfloor.$  It is easily
checked by varying parameters that an optimal solution of this problem has $x_1
= ... = x_\ell = 1-\delta$, and $x_{\ell+1} = 1 - \ell (1-\delta)$, $x_i = 0$
for $i > \ell+1$, which yields the maximum
\[
 \ell (1-\delta)^k + (1- \ell(1-\delta))^k \leq
(1-\delta)^{k-1}
\]
$(1-\delta)^{k-1}$ being the solution of the continuous analogue of the
optimization problem 
\begin{align*}
 \text{maximize}:& \qquad \|f\|_{L^k(\bR)}^k\\
 \text{subject to}:& \qquad \|f\|_{L^1(\bR)} =1, \; \|f\|_{L^\infty(\bR)} \leq 1-\delta,
\end{align*}
which is less constrained.
\begin{lemma}\label{calculus_lemma}
 Let $k \geq 2$ and let $r = n-a_1 -\frac{1}{2}$.  Assume $a_1
\geq b_1$.  There exists a $c_0 > 0$ such that if $r \leq c_0 n$ then
 \[
 \sum \frac{a_i^k}{n^k} + \sum \frac{b_i^k}{n^k} \leq e^{-\frac{kr}{n} +
\frac{kr^2}{n^2}}.
\]
For all $r \leq n$ we have 
\[
 \sum \frac{a_i^k}{n^k} + \sum \frac{b_i^k}{n^k} \leq e^{-\frac{kr}{2n}}.
\]
\end{lemma}
\begin{proof}
  Set $\delta =\frac{r}{n} \leq c_0$. We may assume that $c_0 \leq
\frac{1}{2}$.  Then the first bound for the maximum reduces to $\delta^k +
(1-\delta)^k$ and we require the statement
\[
\forall\; 0 \leq \delta \leq c_0, \quad (\delta^k + (1-\delta)^k)^{\frac{1}{k}}
\leq e^{-\delta + \delta^2}.
\]
Write the left hand side as $(1-\delta) \left(1 +
\left(\frac{\delta}{1-\delta}\right)^k\right)^{\frac{1}{k}}$.  Then it is
readily checked with calculus that this is decreasing in $k$, hence maximized at
$k = 2$. Since
\[
(1-\delta)^2 + \delta^2 = 1 - 2\delta + 2 \delta^2, \qquad e^{-2\delta +2
\delta^2} = 1-2\delta +4\delta^2 + O(\delta^3)
\]
it follows that the left is bounded by the right for $\delta$ sufficiently
small.

For the second statement, we use the second bound of the maximum, so that we
need to show $(1-\delta)^{k-1} \leq e^{-\frac{k\delta}{2}}$.  The worst case is
$k = 2$, which reduces to the true inequality $(1-\delta) \leq e^{-\delta}$ for
$0 \leq \delta \leq 1$.
\end{proof}

We can now prove the case $2 \leq k \leq 6\log n$ of Proposition
\ref{char_ratio_prop}.

\begin{proof}[Proof of Proposition \ref{char_ratio_prop} for $k \leq 6\log n$]
As usual, assume $a_1 \geq b_1$.  As in the proof of the
Proposition in the case $k > 6 \log n$, we write $\lambda_1 = a_1 + \frac{1}{2}
= n-r$.

For $r \leq n^{\frac{5}{6}}$ we appeal to part a. of Theorem
\ref{character_ratio_theorem}, with the main term $\MT$ from Lemma
\ref{MT_lemma}.
 This gives
 \[
  \left|\frac{\chi^\lambda(C)}{f^\lambda}\right| \leq e^{\frac{-kr}{n}} + O
\left(\exp\left(-k\left(\log \frac{n}{r} +O(1)\right)\right)\right).
 \]
In this range, $\exp(-\frac{kr}{n}) = 1 - \frac{kr}{n} +
O\left(\left(\frac{kr}{n}\right)^2\right) = 1-o(1)$, so that, since $k \geq 2$,
the error term is
negligible, and the condition of the first criterion lemma, Lemma
\ref{the_target_lemma}, is satisfied.

Now suppose $r > n^{\frac{5}{6}}$.  
Let $c_0$ be the constant from Lemma \ref{calculus_lemma}.  Let $c_1$ be a
constant, such that, for $r \leq c_1 n$, $\frac{r}{\log r} \leq \frac{n}{2\log
n}$. Notice that this implies $\frac{r^2}{n^2} \leq \frac{r \log r}{2n \log n}$.
 For $r \leq \min(c_0, c_1) n$, Lemma \ref{calculus_lemma} implies that 
\[
 \log\left(\sum \frac{a_i^k}{n^k} + \sum \frac{b_i^k}{n^k}\right) \leq
\frac{-kr}{n} + \frac{k
r^2}{n^2} \leq \frac{-kr}{n} + \frac{kr\log r}{2n\log n}, 
\]
so that this quantity is bounded by the first term in the maximum of the Small
$k$ criterion lemma, Lemma
\ref{condition_lemma}. Thus
we may assume that $r > \min(c_0, c_1)n$.  In this case, using that $\log r =
\log n + O(1)$, the second bound of
Lemma \ref{calculus_lemma} implies that for a sufficiently large constant $c$
\[
 \log\left(\sum \frac{a_i^k}{n^k} + \sum \frac{b_i^k}{n^k}\right) \leq
\frac{-kr}{2n} \leq
\frac{-kr}{n} + \frac{kr \log r}{2 n \log n} + \frac{ c kr}{n \log n}.
\]
Thus again this is bounded by the first term in the maximum of Lemma
\ref{condition_lemma}.

\end{proof}

\appendix

\section{Proof of the lower bound in Conjecture
\ref{main_conjecture}}\label{lower_bound_appendix}
In this appendix we show a proof of the lower bound in Conjecture
\ref{main_conjecture} for any conjugacy class $C$ on $S_n$ having $k = k_n =
o(n)$ non-fixed points. 

Let the number of 2-cycles in $C$ be 
$j\leq \frac{k}{2}$. The proof goes by comparing the distribution
of
the number of fixed points in a randomly chosen permutation, chosen either
according to
 uniform measure or to $\mu_C^{*t}$. 
 
 Note that expectation against uniform measure at step $t$ is the same as
calculating the $L^2$ inner product with $\chi^n +\sgn(C)^t
\chi^{1^n}$.  By the formulas in
(\ref{small_chars}), $\chi^{n-1,1}(\sigma)$ is one less than the number of
fixed points in $\sigma$, and 
\begin{align*}\E_{U_t}(\chi^{n-1,1}) &= 0\\
 \E_{U_t}((\chi^{n-1,1})^2) &=
\E_{U_t}(\chi^n +
\chi^{n-1,1} + \chi^{n-2, 2} + \chi^{n-2,1,1}) = 1.
\end{align*}
Thus according to uniform measure $\chi^{n-1,1}$ has mean 0 and variance
$1$. 

With respect to $\mu_C^{*t}$ we readily check that for any $\lambda \vdash n$,
(see (\ref{fourier_expansion}) and (\ref{convolution}))
\[
 \E_{\mu_C^{*t}}(\chi^\lambda) = f^\lambda
\left(\frac{\chi^{\lambda}(C)}{f^\lambda}\right)^{t}.
\]
Thus,
\begin{align*}
  \E_{\mu_C^{*t}}(\chi^{n-1,1}) &= \sum_{\sigma \in S_n}
\mu_C^{*t}(\sigma)\chi^{n-1,1}(\sigma) =
f^{n-1,1}\left(\frac{\chi^{n-1,1}(C)}{f^{n-1,1}}\right)^t \\&= (n-1)\left(1 -
\frac{k}{n-1}\right)^t
\end{align*}
and similarly the second moment generates contributions from $\chi^n$,
$\chi^{n-1,1}$, $\chi^{n-2,2}$ and $\chi^{n-2,1,1}$,
\begin{align*}
 \E_{\mu_C^{*t}}&((\chi^{n-1,1})^2) = 1 + (n-1)\left(1-\frac{k}{n-1}\right)^t
\\&+ \frac{(n-1)(n-2)}{2}\left(\frac{(n-1-k)(n-2-k)-2j}{(n-1)(n-2)}\right)^t\\&
+ \frac{(n-1)(n-2)-2}{2}\left(\frac{(n-1-k)(n-2-k)+2j-2}{(n-1)(n-2)-2}\right)^t.
\end{align*}
Obviously
\[
 \left(\frac{(n-1-k)(n-2-k)-2j}{(n-1)(n-2)}\right)^t \leq \left(1 -
\frac{k}{n-1}\right)^{2t}.
\]
Also, 
since $2j \leq k$, 
\[
 \left(\frac{(n-1-k)(n-2-k)+2j-2}{(n-1)(n-2)-2}\right)^t \leq \left(1 -
\frac{k}{n-1}\right)^{2t} \exp\left(O\left(\frac{tk}{n^2}\right)\right).
\]

For $t = (1-\varepsilon)\frac{n}{k}\log n$, $\varepsilon > 0$ we deduce
\begin{align*}
 &\E_{\mu_C^{*t}}(\chi^{n-1,1}) = \exp\left(\varepsilon \log n + O\left(\frac{k
\log n}{n}\right)\right),\\
 &\Var_{\mu_C^{*t}} = \E_{\mu_C^{*t}}((\chi^{n-1,1})^2) -
\E_{\mu_C^{*t}}(\chi^{n-1,1})^2 \\& \qquad \qquad \leq 1+
\E_{\mu_C^{*t}}(\chi^{n-1,1}) \\ & \qquad\qquad\qquad+ (n-1)^2 \left(1 -
\frac{k}{n-1}\right)^{2t}\left(\exp\left(O\left(\frac{tk}{n^2}\right)\right) -
1\right)\\& \qquad \qquad \leq 1+\E_{\mu_C^{*t}}(\chi^{n-1,1}) +
O\left(\frac{\log n}{n^{1-2\varepsilon}}\right).
\end{align*}
Let 
\[
A = \left\{\sigma \in  S_n: \chi^{n-1,1}(\sigma) >
\left(\E_{\mu^{*t}_C}(\chi^{n-1,1}) \right)^{\frac{1}{2}}\right\}.
\]  
Applying Chebyshev's inequality, first with measure $U_t$ and then with measure
$\mu^{*t}_C$, we deduce that
\begin{align*}
\|\mu^{*t}_C - U_t\|_{T.V.} &\geq  P_{\mu^{*t}_C}(A) - P_{U_t}(A) \\&\geq 1 -
(2+o(1))
\exp\left(-\varepsilon \log n + O\left(\frac{k \log n}{n}\right)\right),
\end{align*}
which completes the proof.  Note that if $k = o\left(\frac{n}{\log n}\right)$
then we may take $\varepsilon$ on the scale of $\frac{1}{\log n}$.

\section{The bound for dimension}\label{dimension_bound_appendix}

Recall the dimension formula ($\mu_i = \lambda_i + n - i$)
\[
 f^\lambda = \frac{n!}{\mu_1!\mu_2!...\mu_n!}\prod_{i < j} (\mu_i - \mu_j).
 \]
In our character ratio bounds we set apart the terms pertaining to
$\lambda_1$ to find
\[ 
 f^\lambda = \binom{n}{\lambda_1}\prod_{j=2}^n \frac{\lambda_1-\lambda_j +
j-1}{\lambda_1 + j-1} f^{\lambda-\lambda_1}.
\]
Of course, one could extract more rows from the partition $\lambda$, and, since
$f^\lambda = f^{\lambda'}$, removing columns is also possible.  This makes
plausible the following useful result of Larsen-Shalev concerning
dimensions.
\begin{theorem}[\cite{larsen_shalev} Theorem 2.2]\label{larsen_shalev_theorem}
 Write $\lambda \vdash n$ in Frobenius notation with diagonal of length $m$, and
set $a_i' = \lambda_i - i = a_i - \frac{1}{2}$, $b_i' =\lambda_i'-i = b_i -
\frac{1}{2}$.  Then
\[ 
 \log f^\lambda = (1 + o(1)) \log D(\lambda)
\]
where 
\[
 D(\lambda) = \frac{(n-1)!}{\prod_{i = 1}^m a_i'!b_i'!}.
\]
The error term holds as $n \to \infty$ and is independent of the partition
$\lambda\vdash n$.
\end{theorem}

Using this theorem, we now prove Proposition \ref{dimension_prop}.

\begin{proof}[Proof of Proposition \ref{dimension_prop}]
 In this proof we suppress the $\prime$ on $a_i$ and $b_i$.  Also, we write
$\bN^{\underline{m}}$ to indicate strictly decreasing length $m$ vectors of
non-negative
natural numbers, and given such a vector $\underline{a}$ we
write $|\underline{a}| = \sum a_i$.

Recall that we intend to show that for a sufficiently large fixed $c > 0$,
uniformly in $n$ we have the estimate
\[
 \sum_{\lambda \vdash n} (f^\lambda)^{-\frac{c}{\log n}} = O(1).
\]
In view of Theorem
\ref{larsen_shalev_theorem}, it suffices to prove instead that, for a possibly
larger value of $c$,
\[ \sum_{\lambda \vdash n} D(\lambda)^{-\frac{c}{\log n}}=
(n-1)!^{-\frac{c}{\log n}} \sum_{m =
1}^{\sqrt{n}} \sum_{\substack{\underline{a}, \underline{b} \in
\bN^{\underline{m}}\\ |\underline{a}|+ |\underline{b}| =
n-m}}\prod_{i=1}^m(a_i!b_i!)^{\frac{c}{\log n}} = O(1).\] 
By Stirling's approximation, we reduce to showing that (the $\prime$ on the sum
indicates that if either $a_m = 0$ or $b_m = 0$, its contribution is
excluded)
\begin{align*}
 &\sum_{m=1}^{\sqrt{n}} \sum_{\substack{\underline{a}, \underline{b} \in
\bN^{\underline{m}}\\ |\underline{a}|+ |\underline{b}|\\= n-m}}
\exp\left(\frac{c}{\log n}{\sum_{i=1}^m}' \left(\left(a_i +
\frac{1}{2}\right)\log
\frac{a_i}{e} +
\left(b_i + \frac{1}{2}\right)\log \frac{b_i}{e} + O(1)\right)\right)\\ &
\qquad\ll
\exp\left(\frac{c}{\log n}\left(\left(n -\frac{1}{2}\right) \log n -n +
O(1)\right)\right).
\end{align*}
Use that $\sum(a_i + \frac{1}{2}) + \sum(b_i + \frac{1}{2}) = n$ to cancel the
$-n$ term.  Thus we seek a bound of $O(\exp(cn))$ for 
\begin{align}\notag
 &\sum_{m=1}^{\sqrt{n}}\exp\left(\frac{O(cm)}{\log n}\right)
\\\label{ab_sum}&{\sum_{\substack{\underline{a}, \underline{b} \in
\bN^{\underline{m}}\\ |\underline{a}|+ |\underline{b}| = n-m}}}'
\exp\left(\frac{c}{\log n}\sum_{i=1}^m \left(\frac{2a_i + 1}{2}\log a_i +
\frac{2b_i + 1}{2}\log b_i\right)\right).
\end{align}

By symmetry we may assume that $a_1 \geq b_1$.  Set $a_1+1 =
\lambda_1 = M$, and $\Delta = n-m^2 -(M-m)$. Thus $\Delta$ is the number
of
boxes of partition $\lambda$ that are neither contained in the first row, nor
contained in the square determined by the diagonal. We consider
separately
the cases $\Delta \leq n^{\frac{3}{4}}$ and $\Delta > n^{\frac{3}{4}}$.

First consider $\Delta \leq n^{\frac{3}{4}}$.  Evidently $\max(b_1, a_2) \leq m
+
\Delta$ so that, for fixed $\underline{a}, \underline{b}$ satisfying $\Delta
\leq
n^{\frac{3}{4}}$, the sum
inside the exponential of (\ref{ab_sum}) is bounded by
\begin{align*}
 &\left(M-\frac{1}{2}\right)\log (M-1) + \log(m+\Delta)\left(\sum_{j=2}^m
\left(a_j + \frac{1}{2}\right) +
\sum_{j=1}^m
\left(b_j+\frac{1}{2}\right) \right) \\& \quad\leq (n-m^2 + m -
\Delta)\log(n-m^2 + m-\Delta)\\ & \qquad\qquad \qquad
\qquad\qquad\qquad\qquad\qquad + (m^2
- m
+\Delta)\log(m+\Delta).
\end{align*}
Thinking of $a_1$ as fixed, we wish to bound the number of possible strings
$a_2, ..., a_m$, $b_1, ...,
b_m$.  To do this, observe that the numbers $\lambda_2 - m, \lambda_3-m, ...,
\lambda_m - m$ together with $\lambda_1' -m, \lambda_2'-m, ..., \lambda_m'-m $
are a pair of partitions of combined length $\Delta$.  From the well-known
bound for $p(L)$ the number of partitions of an integer $L$, it follows that the
number
of possible choices for $\lambda_2, ..., \lambda_m$, $\lambda_1', ...,
\lambda_m'$ (or, equivalently $a_2, ..., a_m, b_1, ..., b_m$) is at most
\begin{align*}
 \sum_{\Delta_1 + \Delta_2 = \Delta} p(\Delta_1)p(\Delta_2) &\leq \sum_{\Delta_1
+ \Delta_2 = \Delta} \exp\left(O\left(\sqrt{\Delta_1}+
\sqrt{\Delta_2}\right)\right)\\& \leq
\exp\left(O\left(\sqrt{\Delta}\right)\right).
\end{align*}
  
We can now bound the contribution to (\ref{ab_sum}) from $\Delta
\leq n^{\frac{3}{4}}$. It
is
bounded by
\begin{align*}
 &\sum_{m=1}^{\sqrt{n}}\exp\left(\frac{O(cm)}{\log n}\right)\\ &
\quad \sum_{\Delta =
0}^{n^{\frac{3}{4}}}
\exp\biggl(O(\sqrt{\Delta}) +\frac{c}{\log n}\Bigl((n-m^2 +
m-\Delta)\log(n-\Delta)\\&\qquad \qquad\qquad\qquad +
(m^2-m+\Delta)\log(m+\Delta)\Bigr)\biggr).
\end{align*}
Exchanging the order of summation, this has the bound
\begin{align*}
 &\sum_{\Delta = 0}^{n^{\frac{3}{4}}}\exp\left( \frac{c}{\log n}
(n-\Delta)\log(n-\Delta) + O(\sqrt{\Delta})\right) 
\\
& \qquad \qquad \sum_{m=1}^{\sqrt{n}} \exp\left(\frac{c}{\log n}\left((-m^2 +
m)\log \frac{n-\Delta}{m+\Delta} + O(m)\right)\right)
\\& \ll \sum_{\Delta = 0}^{n^{\frac{3}{4}}}\exp\left( \frac{c}{\log n}
(n-\Delta)\log(n-\Delta) + O(\sqrt{\Delta})\right)  \ll \exp(cn).
\end{align*}

When $\Delta > n^{\frac{3}{4}}$ we bound the sum inside the exponential of
(\ref{ab_sum})
crudely.  Notice that $\max(a_i, b_i) \leq a_1 < n-\Delta$, so that  $\log a_i,
\log b_i \leq \log (n-
n^{\frac{3}{4}})$.  Thus the sum inside the exponential of (\ref{ab_sum})
is bounded by
\[
 \log(n-n^{\frac{3}{4}}) \sum_{i =1}^m \left(\frac{2a_i + 1}{2} + \frac{2b_i +
1}{2}\right) = n \log n - n^{\frac{3}{4}} + o(n^{\frac{3}{4}}).
\]
The total number of such possible $a_i, b_i$ is at most the number of partitions
of $n$, which is $\exp(O(\sqrt{n}))$.  It follows that the contribution of
$\Delta >
n^{\frac{3}{4}}$ to (\ref{ab_sum}) is at most
\[
 \ll \exp\left(cn - c\frac{n^{\frac{3}{4}}}{\log n} +
o\left(\frac{n^{\frac{3}{4}}}{\log n}\right)\right).
\]
Combining this bound with the bound for small $\Delta$ proves the proposition.

\end{proof}

\section{Contour formula for characters of $S_n$}\label{contour_appendix}

We conclude by deriving the extension of Frobenius' character formula to
character ratios of arbitrary conjugacy classes on $S_n$.  This derivation is
modeled upon the derivation of Frobenius' formula from \cite{macdonald}, p. 118.

In this appendix we write partitions in the standard row
notation, so for $\lambda \vdash n$,  $\lambda = (\lambda_1, \lambda_2,
..., \lambda_n)$ counts the number of boxes in each row.  Also 
$\delta = (n-1, n-2, ..., 0)$ and $\mu = \lambda + \delta = (\lambda_1 + n-1,
\lambda_2 + n-2, ..., \lambda_n),$ and  $\mu! =
\mu_1!\mu_2!...\mu_n!$.

Set $\ve_i$ for the standard basis
vector on $\bR^n$. Let $\vx = (x_1, ..., x_n)$ and denote the
Vandermonde of $n$ variables by
\[
 \Delta(\vx) = \prod_{1 \leq i < j \leq n} (x_i - x_j) = \sum_{\sigma \in S_n}
\sgn(\sigma) x_1^{\sigma(n)-1}...x_n^{\sigma(1)-1}.
\]

Let $\vlambda, \vrho \vdash n$ be two
partitions of $n$ with $\vlambda$ indexing
an $f^\vlambda$ dimensional  representation of $S_n$,  and $\vrho$
indexing a conjugacy class.  
Following \cite{macdonald}, we start from the
classical fact that the character $\chi_\rho^\lambda$ is equal to the
coefficient of the monomial $\vx^\mu = x_1^{\mu_1}...x_n^{\mu_n}$ in
\[
 \left(\sum_{i=1}^n x_i^{\rho_1}\right) \left(\sum_{i=1}^n
x_i^{\rho_2}\right)...\left(\sum_{i=1}^n x_i^{\rho_n}\right) \Delta(\vx).
\]
The key calculation of \cite{macdonald} is that, for any $m \geq 1$ and for
any $\mu$ as above, 
\[
 \coeff_{\vx^\mu} \left(\sum_{i=1}^n x_i\right)^m \Delta(\vx) = \frac{m!}{\mu!}
\Delta(\mu).
\]
The calculation is as follows.  Expanding $\Delta$,
\begin{align*}
 \coeff_{\vx^\mu} \left(\sum_{i=1}^n x_i\right)^m \Delta(\vx) &= \sum_{\sigma
\in S_n} \sgn(\sigma) \coeff_{\vx^{\mu - \sigma + 1}} \left(\sum_{i=1}^n
x_i\right)^m\\
&= \sum_{\sigma
\in S_n} \sgn(\sigma) \binom{m}{\mu - \sigma + 1}
\\ &= \frac{m!}{\mu!} \det \left[\frac{\mu_i!}{(\mu_i - \sigma_j +
1)!}\right]\\ &= \frac{m!}{\mu!}
\det\left[\mu_i^{\underline{\sigma_j-1}}\right] = \frac{m!}{\mu!}\Delta(\mu).
\end{align*}

In particular, this gives the dimension formula,
\[
 f^\lambda = \frac{n!}{\mu!} \Delta(\mu)
\]
and, when $\rho = (k, 1^{n-k})$ is the class of a single $k$-cycle, 
\begin{equation}\label{k_cycle_formula}
 \chi^\lambda_{(k, 1^{n-k})} = (n-k)! \sum_{i=1}^n \frac{\Delta(\mu - k
\ve_i)}{(\mu - k \ve_i)!}.
\end{equation}
Here, a term $i$ with $k > \mu_i$ is to be omitted.  Thus
\begin{equation}\label{k_cycle_ratio}
 \frac{\chi^\lambda_{(k, 1^{n-k})}}{f^\lambda} = \frac{1}{n^{\underline{k}}}
\sum_{i=1}^n \frac{\mu_i^{\underline{k}} \Delta(\mu - k\ve_i)}{
\Delta(\mu)}.
\end{equation} 
Expanding the Vandermonde, one may check by inspection that the sum of 
(\ref{k_cycle_ratio}) is equal to the sum of the (finite) residues of the
meromorphic function 
\begin{equation*}
\frac{G_k^\mu(z)}{n^{\underline{k}}} :=
\frac{-z^{\underline{k}}}{kn^{\underline{k}}} \prod_{i=1}^n
\frac{z-\mu_i-k}{z-\mu_i}.
\end{equation*}
This last fact is the version of Frobenius' formula derived in
\cite{macdonald}. The equivalence between this formulation from \cite{macdonald}
and the integral formula used throughout the present paper,
\[
 \frac{\chi^{\lambda}_{(k, 1^{n-k})}}{f^\lambda} = \frac{-1}{k
n^{\underline{k}}}\oint F_k^{a,b}(z) dz,
\]
follows from the identity \[G_k^\mu(z) = \frac{-F_k^{a,b}\left(z-n -
\frac{k-1}{2}\right)}{k},\] see e.g. \cite{fulton_harris} pp. 51--52.

Now consider a conjugacy class with non-trivial cycles \[\vk = k_1 \geq k_2 \geq
... \geq k_r>1, \qquad \sum k_i = k,\] so that $\rho = (k_1, ..., k_r,
1^{n-k})$. The appropriate generalization of (\ref{k_cycle_formula}) and
(\ref{k_cycle_ratio}) are
\begin{equation}\label{general_formula}
 \chi^\lambda_{(k_1, ..., k_r, 1^{n-k})} = (n-k)! \sum_{1 \leq i_1, ..., i_r
\leq n} \frac{\Delta(\mu - k_1 \ve_{i_1} - ... - k_r \ve_{i_r})}{(\mu - k_1
\ve_{i_1} - ... - k_r \ve_{i_r})!},
\end{equation}
again with the convention that a term with negative coordinate is to be
omitted.  It follows 
\begin{equation}\label{general_ratio}
 \frac{\chi^\lambda_{(k_1, ..., k_r, 1^{n-k})}}{f^\lambda} =
\frac{1}{n^{\underline{k}}} \sum_{1 \leq i_1, ..., i_r \leq n}
\frac{\mu^{\underline{\sum k_j \ve_{i_j}}}
\Delta(\mu -\sum k_j \ve_{i_j} )}{
\Delta(\mu)}.
\end{equation} 
The formula (\ref{general_formula}) follows, as does (\ref{k_cycle_formula}),
by expanding
\begin{align*}
 &\coeff_{\vx^\mu} \left(\sum x_i^{k_1}\right)\left(\sum
x_i^{k_2}\right)...\left(\sum x_i^{k_r}\right) \left(\sum x_i\right)^{n-k}
\Delta(\vx) \\&= \sum_{1 \leq i_1, ..., i_r \leq n} \coeff_{\vx^{\mu - \sum_j
k_j \ve_{i_j}}} \left(\sum x_i\right)^{n-k} \Delta(\vx).
\end{align*}

We now check that
(\ref{general_ratio}) may be evaluated by a multiple contour integral.

\begin{theorem}
 Let $\lambda \vdash n$ index an irrep. of $S_n$, and let $\vk = k_1 \geq k_2
\geq ... \geq k_r$, $k = \sum k_i \leq n$ denote the non-trivial cycles in
conjugacy class $C$.  Then 
\begin{equation}\label{general_integral_formula}
 \frac{\chi^\lambda(C)}{f^\lambda} = \frac{1}{n^{\underline{k}}}
\left(\frac{1}{2\pi i}\right)^r \oint\cdots \oint_{\cC_1 \times ...
\times \cC_r} G_\vk^\mu(\vz) d\vz,
\end{equation}
with
\[
 G_\vk^\mu(\vz) = \prod_{i=1}^r G_{k_i}^\mu(z_i) \times \prod_{1 \leq i < j \leq
r} \frac{(z_i - z_j) (z_i -z_j -k_i + k_j)}{(z_i - z_j - k_i)(z_i - z_j + k_j)}
\]
and each $G_{k_i}^\mu$ given by
\[G_{k_i}^\mu(z_i) =
\frac{-z^{\underline{k_i}}}{k_i} \prod_{j=1}^n
\frac{z_i-\mu_j-k_i}{z_i-\mu_j}.\]
  The contours $\cC_1,
..., \cC_r$ are to be chosen  such that $\cC_1$ has winding
number 1 about each pole of $G_{k_1}^\mu(z_1)$, and in general, $\cC_i$ has
winding number 1 about each pole of $G_{k_i}^\mu(z_i)$ and also, viewed
as loops in a single complex plane, encloses $(\cC_j +
k_i) \cup (\cC_j - k_j)$ for all $1 \leq j < i$.
\end{theorem}

\begin{proof}
We  check that
\begin{equation*}
  \sum_{1 \leq i_1, ..., i_r \leq n}
\mu^{\underline{\sum k_j \ve_{i_j}}}
\frac{\Delta(\mu -\sum k_j \ve_{i_j} )}{
\Delta(\mu)}= 
\left(\frac{1}{2\pi i}\right)^r \oint\cdots \oint_{\cC_1 \times ...
\times \cC_r} G_\vk^\mu(\vz) d\vz
\end{equation*}
holds for arbitrary $\mu \in \bC^n$, which immediately gives the Theorem.

The case
$r = 1$ is Frobenius' formula.
Suppose that $r \geq 2$ and that the formula holds in case $r-1$.  Let
$\hat{\vk}_1$ to indicate $k_2, ..., k_{r}$ with $k_1$
omitted and similarly $\hat{\vz}_1$.  Write 
\begin{align*}
&\sum_{1 \leq i_1, ..., i_r \leq n}
\frac{\mu^{\underline{\sum k_j \ve_{i_j}}}
\Delta(\mu -\sum k_j \ve_{i_j} )}{
\Delta(\mu)} =  \sum_{1 \leq i_1 \leq
n}
\mu_{i_1}^{\underline{k_1}} \frac{\Delta(\mu - k_1
\ve_{i_1})}{\Delta(\mu)}\times\\& \qquad\qquad\qquad\qquad \times
\sum_{1 \leq i_2, ..., i_r \leq n}\frac{(\mu - k_1
\ve_{i_1})^{\underline{\sum_{j >1
} k_j \ve_{i_j}} }\Delta(\mu - \sum_{j=1}^r k_j
\ve_{i_j})}{\Delta(\mu - k_1 \ve_{i_1})}.
\end{align*}
Invoking the
inductive assumption, the RHS becomes
\begin{align}\label{sum_integral}
& \sum_{1 \leq i_1 \leq n}
\left[z_1^{\underline{k_1}}\prod_{j \neq i_1} \frac{z_1 - \mu_j -
k_1}{z_1- \mu_j}\right]_{z_1 = \mu_{i_1}}\\& \qquad \times \left(\frac{1}{2\pi
i}\right)^{r-1} \oint\cdots \oint_{\cC_2 \times ... \times \cC_{r}} G_{\mu - k_1
\ve_{i_1}}^{\hat{\vk}_1} (\hat{\vz}_1 )dz_2 ...dz_{r}. \notag
\end{align}

Now write 
\[
 G_{\mu - k_1
\ve_{i_1}}^{\hat{\vk}_1} (\hat{\vz}_1 ) = G_{\mu}^{\hat{\vk}_1}(\hat{\vz}_1)
\left[\prod_{j=2}^r\frac{(z_1 - z_j)(z_1 - z_j -k_1 + k_j)}{(z_1
- z_j + k_j)(z_1 -z_j - k_1)}\right]_{z_1 = \mu_{i_1}}
\]
and observe that
\[
 G_\mu^{\vk}(\vz) = G_{\mu}^{\hat{\vk}_1}(\hat{\vz}_1) G_\mu^{k_1}(z_1)
\prod_{j=2}^r\frac{(z_1 - z_j)(z_1 - z_j -k_1 + k_j)}{(z_1
- z_j + k_j)(z_1 -z_j - k_1)}.
\]

Since the poles of $G_{\mu}^{k_1}(z)$ are at $\mu_1, ..., \mu_n$, 
for any fixed $z_2, ..., z_r$, for any choice of contour $\cC_1$ having
winding number 1 about the poles of $G_\mu^{k_1}(z_1)$ and winding number zero
about $\{z_j - k_j, z_j + k_1\}_{2 \leq j \leq r}$,
\[
 \frac{1}{2\pi i}\oint_{\cC_1}  G_\mu^{\vk}(\vz) dz_1 = \sum_{i_1 = 1}^n 
\left[z_1^{\underline{k_1}}\prod_{j \neq i_1} \frac{z_1 - \mu_j -
k_1}{z_1- \mu_j}\right]_{z_1 = \mu_{i_1}}G_{\mu - k_1
\ve_{i_1}}^{\hat{\vk}_1} (\hat{\vz}_1 ).
\]
In view of the restriction on $\cC_1, ..., \cC_r$, the identity follows.
\end{proof}

\bibliographystyle{plain}

\end{document}